%% file: Publication_Graph_induced_subG_of_Johnson_graphs.tex
\documentclass[11pt,reqno]{amsart}

\textfloatsep = 0.4in \addtocontents{toc}{\vspace{0.4in} \hfill
Page\endgraf} \addtocontents{lof}{\vspace{0.2in} \hspace{0.13in} \
Figure\hfill Page\endgraf} \addtocontents{lot}{\vspace{0.2in}
\hspace{0.13in} \ Table\hfill Page\endgraf}

\usepackage{array}
\usepackage{pdfpages}
\usepackage{oplotsymbl}
\usepackage{todonotes}
\usepackage{yfonts}
\usepackage{hyperref}
\usepackage{pgfplots}
\usepgfplotslibrary{groupplots}
\pgfplotsset{compat=1.6}
\pgfplotsset{every axis title/.append style={at={(0.6,1.1)}}}
\usepackage{tcolorbox}
\tcbuselibrary{skins}
\usepackage{algorithm}
\usepackage{algorithmic}
\usepackage{faktor}
\usepackage{array} 
\usepackage{pgf, tikz} 	
\usetikzlibrary{automata, positioning, arrows, decorations.pathreplacing, decorations.pathmorphing}

\definecolor{brightlavender}{rgb}{0.75, 0.58, 0.89}
\definecolor{amethyst}{rgb}{0.6, 0.4, 0.8} 	
\definecolor{blue-violet}{rgb}{0.54, 0.17, 0.89} 	

\definecolor{burgundy}{rgb}{0.5, 0.0, 0.13} 	
\definecolor{burntorange}{rgb}{0.93, 0.53, 0.18} 	
\definecolor{earthyellow}{rgb}{0.88, 0.66, 0.37}
\definecolor{applegreen}{rgb}{0.55, 0.71, 0.0}	
\definecolor{antiquefuchsia}{rgb}{0.57, 0.36, 0.51}
\definecolor{ultramarine}{rgb}{0.07, 0.04, 0.56}

\usepackage{bbm}

\usepackage[utf8]{inputenc}

\usepackage{graphicx}

\usepackage{hyperref}
\hypersetup{
    colorlinks = true,
    citecolor=blue,
    filecolor=black,
    linkcolor=blue,
    urlcolor=black
}

\usepackage{caption}
\usepackage{subcaption}
\usepackage{tikz-network}
 	
\usepackage{amsthm}
\usepackage{amssymb}
\usepackage{xfrac}
\usepackage{latexsym}
\usepackage{amsbsy}
\usepackage{amsfonts}
\usepackage{appendix}
\usepackage{pstricks}
\usepackage{graphicx}
\usepackage[utf8]{inputenc}
\usepackage{enumitem}

\numberwithin{equation}{section}
\newtheorem{theorem}{Theorem}[section]

\newtheorem{proposition}[theorem]{Proposition}
\newtheorem{lemma}[theorem]{Lemma}
\newtheorem{corollary}[theorem]{Corollary}
\newtheorem{definition}[theorem]{Definition}

\usepackage[backend=bibtex, doi=false, isbn=false, url=false, citestyle=alphabetic]{biblatex}

\begin{document}

\title[Johnson Sub-graphs]{On sub-graphs of the Johnson $\mathcal{J}(n,k)$ graph induced by particle systems on simple graphs.}

\author[J. W. Fischer]{\textbf{\quad {Jens Walter} Fischer$^{\spadesuit\clubsuit}$ \, \, }}
\address{{\bf {Jens Walter} FISCHER},\\ Institut de Math\'ematiques de Toulouse. CNRS UMR 5219. \\
Universit\'e Paul Sabatier
\\ 118 route
de Narbonne, F-31062 Toulouse cedex 09.} \email{jens.fischer@math.univ-toulouse.fr}

\begin{abstract}
	Understanding the topology of the state space has proven to be extremely efficient for dynamical systems with a continuous state space. On the other hand, for particle systems on finite simple graphs, it has not yet been subject to deep investigation due to combinatorial hurdles and existing efficient spectral theoretic approaches to the analysis of classical particle systems. In the context of complex systems with heterogeneous interactions of particles, these techniques can break down due to intractability. This work provides a tool box of results on the topology of state spaces of particles systems under the sole conditions that exactly one particle moves at a time and no two particles may occupy the same vertex at the same time. The Johnson $\mathcal{J}(n,k)$ graph yields the overarching structure and we prove that any particle system may be embedded as a dynamical system on a sub-graph of $\mathcal{J}(n,k)$.    
\end{abstract}
\bigskip

\maketitle

\textsc{$^{\spadesuit}$  Universit\'e de Toulouse}
\smallskip

\textsc{$^{\clubsuit}$ University of Potsdam}
\smallskip

\textit{ Key words : Johnson graph, Intersection Graph Theory, Vertex-induced Sub-graphs, Combinatorics, Complex Systems}  
\bigskip

\textit{ MSC 2010 : } 
\vspace{25pt}

\begin{center}
	\textbf{\large{This is work in progress. It is not necessarily in a perfect form.}} \\ \textbf{\large{Additional results may be added over time.}}\vspace{5pt}
	\\ EDIT 02.11.2022: It came to my attention that identical concepts were introduced in \cite{AUDENAERT200774} as \textit{symmetric $k$th power of a graph} and in \cite{Fabila-Monroy2012} as \textit{Token Graphs}. Our results, while of fundamental nature, are still complementary to what has been discussed in \cite{AUDENAERT200774} and \cite{Fabila-Monroy2012} as well as to work based on those two publications.
\end{center}
\vspace{25pt}

\section{Introduction}
	\input{chapters/kPG_intro.tex}
\section{Definiton and General Properties}
	\input{chapters/kPG_definition.tex}
\section{General Properties of kPG}
	\input{chapters/kPG_general_prop.tex}
\section{kPGs induced by regular graphs}
	\input{chapters/kPG_regular_graph.tex}\newpage
\section{Outlook}
	\input{chapters/kPG_Outlook.tex}
\newpage
\printbibliography

\newpage
\appendix
\section{Special cases for $\bar{d}$-regular graph}
To begin with, we consider the cases $k=2$ and $k=\bar{n}-2$. In these cases due to the isomorphism $\mathfrak{L}_k\stackrel{\sim}{=}\mathfrak{L}_{\bar{n}-k}$ we only have to make claims about one or the other. It turns out that the degree set $D_k$ is given by $D_k=\{2\bar{d},2(\bar{d}-1)\}$ since for $k=2$ the two particles can be neighbors and block, therefore, an adjacent vertex mutually, or not. We define the level sets with respect to the degree as $\mathfrak{V}_{D_k;2l}:=\{\mathfrak{v}\in\mathfrak{V}_k|\,\mathrm{deg}_k(\mathfrak{v})=2l\}$ for $l\in\{\bar{d},\bar{d}-1\}$. Remark that $|\mathfrak{V}_{D_k;2(\bar{d}-1)}|=|E|=\frac{\bar{n}\bar{d}}{2}$. An important question to ask is how large subsets of $\mathfrak{V}_k$ can be chosen when taking into account weights given by the degree sequence on said subset. 
We consider to this end for $\mathfrak{U}\subset\mathfrak{V}_k$ the map
\begin{equation*}
	\mathfrak{f}(\mathfrak{U}):=\dfrac{\bar{d}|\mathfrak{V}_{D_k;2(\bar{d}-1)}\cap\mathfrak{U}|+(\bar{d}+1)|\mathfrak{V}_{D_k;2\bar{d}}\cap\mathfrak{U}|}{\frac{\bar{n}\bar{d}^2}{2}+ \frac{\bar{n}(\bar{n}-1-\bar{d})(\bar{d}+1)}{2}}
\end{equation*}
which represents the weighted sum of the vertices in $\mathfrak{U}$ where each weight corresponds to the quotient of the number of smaller and larger degrees then $\mathrm{deg}_k(\mathfrak{v})$. As there are only two sets $\mathfrak{V}_{D_k;2(\bar{d}-1)}$ and $\mathfrak{V}_{D_k;2\bar{d}}$ which constitutes $\mathfrak{V}_k$ and vertices in $\mathfrak{V}_{D_k;2(\bar{d}-1)}$ have a smaller weight than vertices in $\mathfrak{V}_{D_k;2\bar{d}}$, the size of $|\mathfrak{U}|$ is maximal with subject to the constraint $\mathfrak{f}(\mathfrak{U})\leq 2^{-1}$ if $\mathfrak{U}$ contains as many states with small degree as possible. Indeed, the size of can be characterized by
\begin{equation}
	\max_{\mathfrak{U}\subseteq \mathfrak{V}_k, \mathfrak{f}(\mathfrak{U})\leq 2^{-1}}|\mathfrak{U}| = \begin{cases} \dfrac{\bar{d}\bar{n}}{2} + y_{\ast},& 0\leq \bar{d}\leq \frac{1}{2}\left(\bar{n}-1+\sqrt{\bar{n}-1}\sqrt{\bar{n}-1+4}\right)\\
	y^{\ast} ,& otherwise
	\end{cases}
\end{equation}
where we define
\begin{align*}
	y_{\ast} &= \left\lfloor\dfrac{\bar{n}(\bar{n}-1-\bar{d})}{4}-\dfrac{\bar{d}^2\bar{n}}{4(\bar{d}+1)}\right\rfloor;\quad y^{\ast} = \left\lfloor\dfrac{\bar{d}+1}{\bar{d}}\dfrac{\bar{n}(\bar{n}-1-\bar{d})}{4}+\dfrac{\bar{d}\bar{n}}{4}\right\rfloor.
\end{align*}
We outline a proof in what follows. First assume that 
\begin{equation*}
	\bar{d}\leq \frac{1}{2}\left(\bar{n}-1+\sqrt{\bar{n}-1}\sqrt{\bar{n}-1+4}\right).
\end{equation*}
Recall that $D_k=\{d\in\mathbb{N}| \exists\mathfrak{v}\in\mathfrak{V}_k:\mathrm{deg}_k(\mathfrak{v})=d\}$ and for $l\in D_k$ that $\mathfrak{V}_{D_k;l}:=\{\mathfrak{v}\in \mathfrak{V}_k|\mathrm{deg}_k(\mathfrak{v})=l\}$. In the case $k=2$, there are exactly two distinct degrees in $\mathfrak{V}_k$ and, hence, only two sets $\mathfrak{V}_{D_k;l}$, namely, $\mathfrak{V}_{D_k;2(\bar{d}-1)}$ and $\mathfrak{V}_{D_k;2\bar{d}}$. Indeed, in the discussed case, we can derive $|\mathfrak{U}^{\ast}|:=\max_{\mathfrak{U}\subseteq \mathfrak{V}_k, \mathfrak{f}(\mathfrak{U})\leq 2^{-1}}|\mathfrak{U}|$ explicitly, which becomes important since $|\mathfrak{S}|\leq |\mathfrak{U}^{\ast}|$. Considering the representation of sub-sets of size $2$ of a vertex set as an edge, we obtain that $|\mathfrak{V}_{D_k;2(\bar{d}-1)}|=\frac{\bar{n}\bar{d}}{2}$ and, consequently, $|\mathfrak{V}_{D_k;2(\bar{d}-1)}|=\frac{\bar{n}(\bar{n}-1-\bar{d})}{2}$. We can now consider the two cases for $\max_{\mathfrak{U}\subseteq \mathfrak{V}_k, \mathfrak{f}(\mathfrak{U})\leq 2^{-1}}|\mathfrak{U}|$ which correspond to the cases $\frac{\bar{n}\bar{d}^2}{2}\leq \frac{\bar{n}(\bar{n}-1-\bar{d})(\bar{d}+1)}{2}$ and $\frac{\bar{n}\bar{d}^2}{2}\geq \frac{\bar{n}(\bar{n}-1-\bar{d})(\bar{d}+1)}{2}$, respectively.\par
Consequently, in the first case, the maximal set $\mathfrak{U}$ satisfies $|\mathfrak{V}_{D_k;2(\bar{d}-1)}|\leq|\mathfrak{U}|$ and writing $|\mathfrak{U}|=|\mathfrak{V}_{D_k;2(\bar{d}-1)}|+y=\frac{ \bar{d}\bar{n}}{2}+y$ for some non-negative integer $y$ we obtain that for all $y\leq y_{\ast}$, for $y_{\ast}$ as in the claim, the inequality
\begin{align*}
	\mathfrak{f}(\mathfrak{U})&=\dfrac{\frac{\bar{n}\bar{d}^2}{2}+(\bar{d}+1)y}{\frac{\bar{n}\bar{d}^2}{2}+ \frac{\bar{n}(\bar{n}-1-\bar{d})(\bar{d}+1)}{2}}\vspace{5pt} \leq \dfrac{\frac{\bar{n}\bar{d}^2}{2}+(\bar{d}+1)\left(\frac{\bar{n}(\bar{n}-1-\bar{d})}{4}-\frac{\bar{d}^2\bar{n}}{4(\bar{d}+1)}\right)}{\dfrac{\bar{n}\bar{d}^2}{2}+ \frac{\bar{n}(\bar{n}-1-\bar{d})(\bar{d}+1)}{2}}
\end{align*}
is satisfied where the last term is smaller or equal one half. Let $\mathfrak{A}_{y_{\ast}}\subset \mathfrak{V}_{D_k;2\bar{d}}$ be any subset of size $y_{\ast}$. Then, choosing $\mathfrak{U}=\mathfrak{V}_{D_k;2(\bar{d}-1)}\cup \mathfrak{A}_{y_{\ast}}$ we obtain a subset $\mathfrak{U}\subset \mathfrak{V}_k$ of size $\frac{\bar{d}\bar{n}}{2}$ with $\mathfrak{f}(\mathfrak{U})\leq 2^{-1}$ and any larger subset $\mathfrak{W}$ gives $\mathfrak{f}(\mathfrak{W})> 2^{-1}$. We have therefore found the size maximum of any subset $\mathfrak{U}$ satisfying $\mathfrak{f}(\mathfrak{U})\leq 2^{-1}$. In the second case, we can choose $\mathfrak{U}\subset\mathfrak{V}_{D_k;2(\bar{d}-1)}$ and 
\begin{equation*}
	\mathfrak{f}(\mathfrak{U})=\dfrac{\frac{\bar{n}\bar{d}y}{2}}{\frac{\bar{n}\bar{d}^2}{2}+ \frac{\bar{n}(\bar{n}-1-\bar{d})(\bar{d}+1)}{2}}.
\end{equation*} 
The remaining arguments can be made analogously to the first case. Due to the symmetry \par
We discuss now special cases $\bar{d}\in\{2,\bar{n}-2\}$. In these cases explicit quantitative statements can be made about the structure of $\mathfrak{L}_k$ for any $k\in\{1,\hdots,\bar{n}-1\}$. We lead into this section with the case $\bar{d}=2$. 
We consider first the case $\bar{d}=2$ and analyze the structure of the underlying graph $\mathfrak{L}_k$. This means we are working with a generalized exclusion process on a cycle graph as illustrated in Figure \ref{fig:particles_on_cycle_graph}. We assume that there are $k\leq \frac{\bar{n}}{2}$ particles since by symmetry $\mathfrak{L}_k\stackrel{\sim}{=}\mathfrak{L}_{\bar{n}-k}$ we cover the remaining cases as well. 
\begin{figure}[!htb]
	\centering
	\begin{tikzpicture}[scale = 0.8]
		\Vertex[x=0,y=0]{t1}
		\Vertex[x=-2,y=1]{t2}
		\Vertex[x=-2,y=3]{t3}
		\Vertex[x=-1,y=4]{t4}
		\Vertex[x=1,y=4]{t5}
		\Vertex[x=2,y=3]{t6}
		\Vertex[x=2,y=1]{t7}
	    \Vertex[x=0,y=0,color = blue,size = 0.5]{z}
	    \Vertex[x=-2,y=1,color = blue,size = 0.5]{z}
	    \Vertex[x=2,y=3,color = blue,size = 0.5]{z}
		\Edge(t1)(t2)
		\Edge(t3)(t2)
		\Edge(t3)(t4)
		\Edge(t4)(t5)
		\Edge(t5)(t6)
		\Edge(t6)(t7)
		\Edge(t7)(t1)
	\end{tikzpicture}
	\caption{A particle configuration of $3$ particles on a cycle graph of length $7$.\label{fig:particles_on_cycle_graph}}
\end{figure}
To find the degrees in $\mathfrak{L}_k$ we can write any configuration $\mathfrak{v}$ of particles as a vector which contains as entries the size of a connected component of $\mathfrak{L}_{\mathfrak{v}}$. Assume that there are $c'_{\mathfrak{v}}$ connected components of $\mathfrak{L}_{\mathfrak{v}}$ and denote their respective sizes by $k_i$. Each connected component contains exactly $k_i-1$ edges due to the structure of the underlying cycle graph $L$. Therefore, we obtain that 
\begin{equation*}
	\mathrm{deg}_k(\mathfrak{v}) = 2k-2\sum_{i=1}^{c'_{\mathfrak{v}}} (k_i -1) = 2 c'_{\mathfrak{v}}.
\end{equation*}
Furthermore, since $k\leq \frac{\bar{n}}{2}$ we can construct for any $l \in\{1, \hdots,k\}$ a configuration $\mathfrak{v}$ satisfying $c'_{\mathfrak{v}} = l$. This implies that $D_k = \{2l|l\in\{1,\hdots,k\}\}$. The remaining task consist in finding $|\mathfrak{V}_{D_k;d}|$ for any $d\in D_k$ which is equivalent to finding the number of vertex induced $k$ sub-graphs of $L$ containing exactly $\frac{d}{2}$ edges. Figure \ref{fig:particles_on_cycle_graph} shows that this problem is related to finding all unique configurations of two colored balls on a necklace with a fixed number of each color.
To explain this further, we recall that any configuration $\mathfrak{v}$ can be translated to a vector of length $c'_{\mathfrak{v}}$ where each entry corresponds to the length of one connected component of $L_{\mathfrak{v}}$. We can, hence, identify the problem with finding all distinct colored necklaces of length $c'_{\mathfrak{v}}$ with colorings $\{(x_i,y_i)_{i=1}^{c'_{\mathfrak{v}}}|\sum_{i=1}^{c'_{\mathfrak{v}}}(x_i,y_i)=(k,\bar{n}-k)\}$, a problem solved by P\`olya's enumeration theorem. The number of configurations induced by each distinct necklace is then given by the size of its orbit under the symmetric group $S_{c'_{\mathfrak{v}}}$. 
From this we can derive $|\mathfrak{V}_{D_k;2 c'_{\mathfrak{v}}}|$ which gives the first part of the puzzle.
\begin{lemma}\label{lem:size_D2lk_case_bar_d=2}
	Let $L$ be a cycle on $\bar{n}$ vertices and $k\in\left\{1,\hdots,\left\lfloor\frac{\bar{n}}{2}\right\rfloor\right\}$. Then, for $l\in\{0,\hdots,k\}$ the identity
	\begin{equation}
		|\mathfrak{V}_{D_k;2 l}|=\binom{k-1}{l-1}\binom{\bar{n}-k-1}{l-1}\dfrac{\bar{n}}{l}
	\end{equation}
	is satisfied.
\end{lemma}
\begin{proof}
	First, note that
	\begin{align*}
		\binom{k-1}{l-1}\binom{\bar{n}-k-1}{l-1}\dfrac{\bar{n}}{l} &= \binom{k-1}{l-1}\binom{\bar{n}-k-1}{l-1}\dfrac{\bar{n}-k+k}{l}\\
		&=  \binom{k-1}{l-1}\binom{\bar{n}-k-1}{l-1}\dfrac{k}{l} + \binom{k-1}{l-1}\binom{\bar{n}-k-1}{l-1}\dfrac{\bar{n}-k}{l}\\
		&=  \binom{k}{l}\binom{\bar{n}-k-1}{l-1} + \binom{k-1}{l-1}\binom{\bar{n}-k}{l}.
	\end{align*}
	This underlines the symmetry of the problem in $k$ and $\bar{n}-k$. Based on Figure \ref{fig:particles_on_cycle_graph} we identify the problem with counting necklaces of two colors, which are fully defined by the sets $\mathfrak{v}\in\mathfrak{V}_k$. The colors are assumed to be blue and red, where blue represents an occupied site and red an unoccupied site. We want to calculate the number of distinct blue-red colored necklaces respecting rotational symmetry, i.e., two necklaces $\mathfrak{v},\mathfrak{w}\in \mathfrak{V}_k$ are considered identical if and only if $\mathfrak{v}=\mathfrak{w}$. By fixing one vertex of the cycle, without loss of generality $v_1\in V$, we can consider necklaces starting with a blue beat, continuing then counterclockwise. \\
	Assume that $v_1$ is occupied, i.e, carries a blue beat. Place $k$ blue beats on a line of length $k$ and select $l$ of them without replacement. There are $\binom{k}{l}$. Now, separate $\bar{n}-k$ in $l$ positive parts, i.e, create a vector $(n_1,\hdots,n_l)$ with $n_i\in\mathbb{N}^{\ast}$ and $\sum_{i=1}^l n_i=\bar{n}-k$, which is possible in $\binom{\bar{n}-k-1}{l-1}$ ways. Place $n_i$ red beats after the $i$-th drawn blue beat. This construction gives, consequently, all necklaces starting with a blue beat and, thus, there are $\binom{k}{l}\binom{\bar{n}-k-1}{l-1}$ such necklaces.\\
	By symmetry there are $\binom{k-1}{l-1}\binom{\bar{n}-k}{l}$ necklaces starting with a red beat. Summing both terms yields the claim.     
\end{proof}
It turns out that the case $\bar{d}=\bar{n}-2$ is much more forgiving in terms of the difficulty of deriving $|\mathfrak{V}_{D_k;d}|$ for some fixed $d\in D_k$ which in the previous case demanded the utilization of results from combinatorics. In what follows we can reduce the problem to considerations on the graph complement of a $\bar{d}$-regular graph with $\bar{d}=\bar{n}-2$. In Figure \ref{fig:graph_complement_d_bar_equal_n_bar_minus_two} one can see the translation from the graph $L$ to its complement $L^c$ in the sens that $L^c$ contains exactly the edges which are not contained in $L$ while preserving the vertex set. This renders $L^c$ a disjoint union of paths of length $1$, exhibiting $\frac{\bar{n}}{2}$ connected components.
\begin{figure}[!htb]
	\centering
	\begin{tikzpicture}[scale = 0.8]
		\Vertex[x=0,y=0]{t1}
		\Vertex[x=-2,y=1]{t2}
		\Vertex[x=-2,y=3]{t3}
		\Vertex[x=0,y=4]{t4}
		\Vertex[x=2,y=3]{t5}
		\Vertex[x=2,y=1]{t6}
		\Edge(t1)(t2)
		\Edge(t1)(t3)
		\Edge(t1)(t6)
		\Edge(t1)(t5)
		\Edge(t2)(t3)
		\Edge(t2)(t4)
		\Edge(t2)(t5)
		\Edge(t3)(t4)
		\Edge(t3)(t6)
		\Edge(t4)(t5)
		\Edge(t4)(t6)
		\Edge(t5)(t6)
		\draw[{Latex}-{Latex}, ultra thick] (3,2) -- (6,2);
		\Vertex[x=9,y=0]{p1}
		\Vertex[x=7,y=1]{p2}
		\Vertex[x=7,y=3]{p3}
		\Vertex[x=9,y=4]{p4}
		\Vertex[x=11,y=3]{p5}
		\Vertex[x=11,y=1]{p6}
		\Edge(p1)(p4)
		\Edge(p2)(p6)
		\Edge(p3)(p5)
	\end{tikzpicture}
	\caption{Translation from the $4$-regular graph $L$ on six vertices to the disjoint union of three paths of length $1$, which corresponds to the complement graph of $L$, denoted by $L^c$. \label{fig:graph_complement_d_bar_equal_n_bar_minus_two}}
\end{figure}
\begin{lemma}\label{lem:deg_set_n_m_2}
	Let $L$ be a $\bar{n}-2$ regular graph and $k\in\left\{1,\hdots,\frac{\bar{n}}{2}\right\}$. Then,
	\begin{equation}
		D_k = \left\{k(\bar{n}-2)-k(k-1) + 2l\;\bigg|\;l\in\left\{0,\hdots,\left\lfloor\frac{k}{2}\right\rfloor\right\}\right\} 
	\end{equation}		 
	and, consequently, any degree is even.
\end{lemma}
\begin{proof}
	We employ the idea based on Figure \ref{fig:graph_complement_d_bar_equal_n_bar_minus_two} that it is sufficient to consider the graph complement $L^c$.	Let $d\in D_k$. Then there exists a $\mathfrak{v}\in\mathfrak{V}_k$ such that $\mathrm{deg}_k(\mathfrak{v}) = d$. Since $\mathfrak{v}\subset V$ it also spans a vertex induced sub-graph in $L^c$. Denote by $l$ the number of paths in $L^c$ which are fully contained in $\mathfrak{v}$, i.e., both vertices of the path are elements of $\mathfrak{v}$. Then, $l\in \left\{0,\hdots,\left\lfloor\frac{k}{2}\right\rfloor\right\}$ and 
	\begin{equation}
		d = \mathrm{deg}_k(\mathfrak{v}) = k(\bar{n}-2) -  k(k-1) + 2l.
	\end{equation}
	Since $\bar{n}-2$ is even, $k(k-1)$ is even and so is $2l$ for any positive integer $l$ we obtain that all degrees are even.
\end{proof}
Based on this correspondence between $L$ and $L^c$ we can also derive the size of each $\mathfrak{V}_{D_k;d}$ for $d\in D_k$. Additionally, we can derive $l_{\mathfrak{v}}:=|\{\mathrm{deg}_k(\mathfrak{w})|\mathrm{deg}_k(\mathfrak{w})>\mathrm{deg}_k(\mathfrak{v})\}|$. Indeed, if $\mathrm{deg}_k(\mathfrak{v})=k(\bar{n}-2) -  k(k-1) + 2l$, then $l_{\mathfrak{v}}=\left\lfloor\frac{k}{2}\right\rfloor-l$.
\begin{lemma}
	Let $L$ be a $\bar{n}-2$ regular graph and $k\in\left\{1,\hdots,\frac{\bar{n}}{2}\right\}$. Then, for $j_d\in\left\{0,\hdots,\left\lfloor\frac{k}{2}\right\rfloor\right\}$ the set, corresponding to the degree $	d = k(\bar{n}-2) -  k(k-1) + 2j_d\in D_k$, satisfies the identity 
	\begin{equation}
		|\mathfrak{V}_{D_k;d}|=\sum_{l=0}^k \binom{\frac{\bar{n}}{2}}{l}\binom{l}{j_d}\binom{\frac{\bar{n}}{2}-l}{k-l-j_d}.
	\end{equation}
\end{lemma}
\begin{proof}
	Fix $j_d\in\left\{0,\hdots,\left\lfloor\frac{k}{2}\right\rfloor\right\}$. We write the vertices $v\in V$ in $\frac{\bar{n}}{2}$ columns and two rows where two vertices connected by an edge in $L^c$ belong to the same column. First, place $l$ particles in the first row. We have $\binom{\frac{\bar{n}}{2}}{l}$ possibilities to do so. To obtain a configuration $\mathfrak{v}$ with $\mathrm{deg}_k(\mathfrak{v})=k(\bar{n}-2) -  k(k-1) + 2j_d$ we have to place precisely $j_d$ of the remaining $k-l$ particles in the second row among the neighbors of the already occupied vertices. We have $\binom{l}{j_d}$ to obtain this. Finally, we distribute the remaining $k-l-j_d$ particles on the remaining $\frac{\bar{n}}{2}-l$ vertices in the second row which is possible in $\binom{\frac{\bar{n}}{2}-l}{k-l-j_d}$ ways. To obtain the total, we sum over all possibilities for $l\in\{0,\hdots,k\}$ which leads to the formula
	\begin{equation*}
		|\mathfrak{V}_{D_k;d}|=\sum_{l=0}^k \binom{\frac{\bar{n}}{2}}{l}\binom{l}{j_d}\binom{\frac{\bar{n}}{2}-l}{k-l-j_d}.
	\end{equation*}
\end{proof}
We could demonstrate that for certain special cases of $k$ and $\bar{d}$ it is possible to obtain explicit results on the structure of $\mathfrak{L}_k$ associated to the underlying graph $L$. In particular, we exploited in the cases where we fixed $\bar{d}$, that there is exactly one regular graph with this degree. Deviating from these special cases immediately yields non-uniqueness of the underlying graph $L$. For example, consider the case $\bar{n}=6$ and $\bar{d}=3$. Then, there are already two graphs with vastly different properties, which satisfy this condition. One of them even is bipartite, while the other is not, which has by Proposition \ref{prop:L_bipartite_L_k_bipartite} important influences on the structure of the associated $\mathfrak{L}_k$.
\end{document}

%% file: chapters/kPG_intro.tex
Particle systems yield in the context of complex systems some accessible toy models, which allow for a reduction of the network complexity and disorder induced by the particle configuration due to an imposed crude structure. See for example \cite{10.1214/aoap/1177005359} and the discussion of the exclusion process therein, in particular the case of regular graphs, which are in spite of the richness of this class of graphs, covered by the same result. This breaks down as soon as the structure of these toy models is altered ever so slightly. This work is dedicated to the topology of graphs representing, in particular, the state space of exclusion processes with $k$ indistinguishable particles on a simple finite connected graph $L=(V,E)$. We discuss in various examples their link to random graphs and complex systems, in particular, of the social network type with dynamic relationships between individuals. These state spaces can be identified with another graph $\mathfrak{L}_k$ which represents the possible configurations of particles on the underlying simple graph. We will call $\mathfrak{L}_k$ the $k$ particle graph (kPG) to honour the particle system motivation. The vertex set of $\mathfrak{L}_k$ will be given by the $k$ subsets of $V$ neighborhood relationship in $\mathfrak{L}_k$ will be defined by the symmetric difference of two configurations which corresponds to the natural transition behavior of exclusion processes. This links $\mathfrak{L}_k$ naturally to the Johnson graph $\mathcal{J}(|V|,k)$, as defined for example in \cite{holton_sheehan_1993}, and $\mathcal{J}(|V|,k)\stackrel{=}{\sim}\mathfrak{L}_k$ if $L=L_c$ is the complete graph on $V$. Consider diagram Figure \ref{fig:diagram_L_Lk_Jnk_intro} and the links represented therein. 
\begin{figure}[!htb]
	\centering
	\begin{tikzpicture}[scale = 0.5]
		\Vertex[x=-3,y=3,label={$L$},size = 1.5]{t1}
		\Vertex[x=3,y=3,label={$L_c$},size = 1.5]{t2}
		\Vertex[x=-3,y=-3,label={$\mathfrak{L}_k$},size = 1.5]{t3}
		\Vertex[x=3,y=-3,label={$\mathcal{J}(\bar{n},k)$},size = 1.5]{t4}
		\Edge[Direct, label={$\hookrightarrow$}](t1)(t2)
		\Edge[Direct, label={$\Psi_k$}](t1)(t3)
		\Edge[Direct, label={$\Psi_k$}](t2)(t4)
		\Edge[Direct, label={$\hookrightarrow$}](t3)(t4);
	\end{tikzpicture}
	\caption{Relationships of graphs which will be discussed in what follows. The graph $L_c$ is the complete graph on the same vertex set as $L$. The map $\Psi_k$ lifts the underlying graph $L$ to the kPG using subsets of $V$ and respecting their symmetric difference.\label{fig:diagram_L_Lk_Jnk_intro}}
\end{figure}
It shows dependencies and implications of the properties of one graph and possible deductions and arguments we can formulate, by moving flexibly through the diagram. We start with the formal definition of the kPG, some related graphs, its general properties before moving on to geometric properties and quantitative results under the assumption that $L$ is regular. This assumption turns out to be weaker than it seems since it applies to important contexts in complex systems, for example the underlying state space topology of models dealing with dynamic relationships in any kind of network.

%% file: chapters/kPG_definition.tex
For any finite simple graph $L=(V,E)$ and an integer $k\in\{1,\hdots,|V|-1\}$ we consider in this work the $(L,k)$ induced sub-graph of the Johnson graph $\mathcal{J}(|V|,k)$ where $L$ is the defining structure for the edges in said graph. 
\begin{definition}\label{def:mathfrak_L_k}
	Consider a finite simple graph $L=(V,E)$. Define $\mathfrak{V}_k=\{\mathfrak{v}\subset v|\;|\mathfrak{v}|=k\}$ and $\mathfrak{E}_k=\{\langle\mathfrak{v},\mathfrak{w}\rangle|\mathfrak{v}\triangle\mathfrak{w}=\{v,w\};\langle v,w\rangle\in E\}$. We denote by $\mathfrak{L}_k=(\mathfrak{V}_k,\mathfrak{E}_k)$ the $(L,k)$ induced sub-graph of the Johnson graph $\mathcal{J}(|V|,k)$ and by $\mathrm{deg}_k(\mathfrak{v})$ the degree of $\mathfrak{v}\in\mathfrak{V}_k$ in $\mathfrak{L}_k$. We call $\mathfrak{L}_k$ the $k$-particle graph on $L$.
\end{definition}
A related structure but on the vertex set $V^k$ was considered in \cite{WILSON197486} and \cite{Greenwell1974} motivated by the solvability of the so called $15$-puzzle. This puzzle consist of a quadratic box which can fit $16$ quadratic tiles of adequate size and $15$ small enumerated quadratic tiles which correspond to this size carrying numbers $1$ to $15$. The game then consists in ordering the tiles row by row only by sliding one tile at the time using the remaining free slot, the first row containing the numbers $1$ to $4$, the second one the numbers $5$ to $8$ and so on. One can wonder about the existence of a solution to this task depending on the initial configuration of the tiles. \par
In contrast, to the best of our knowledge, the graph $\mathfrak{L}_k$ as defined in Definition \ref{def:mathfrak_L_k}, based on some underlying graph $L$, has not yet been considered in the literature. We discuss its rich structure and possible implications of a deeper understanding of it on the analysis of, among various others, vertex induced sub-graphs $L$ throughout the remainder of this section. In Figure \ref{fig:construction_L_from_C4} we show how to construct $\mathfrak{L}_k$ from an underlying graph $L$ to illustrate its structure.  
\begin{figure}[!htp]
	\centering
	\begin{tikzpicture}[scale = 0.5]
		\Vertex[x=-2,y=-2,label=$1$]{t1}
		\Vertex[x=-2,y=2,label=$2$]{t2}
		\Vertex[x=2,y=2,label=$3$]{t3}
		\Vertex[x=2,y=-2,label=$4$]{t4}
		\Edge(t1)(t2)
		\Edge(t2)(t3)
		\Edge(t3)(t4)
		\Edge(t1)(t4);
		\Vertex[x=10,y=2,size=1.2,label = {$\{1,2\}$}]{a1}
		\Vertex[x=10,y=-2,size=1.2,label = {$\{1,3\}$}]{a2}
		\Vertex[x=13,y=4,size=1.2,label = {$\{1,4\}$}]{a3}
		\Vertex[x=16,y=-2,size=1.2,label = {$\{2,3\}$}]{a4}
		\Vertex[x=16,y=2,size=1.2,label = {$\{2,4\}$}]{a5}
		\Vertex[x=13,y=-4,size=1.2,label = {$\{3,4\}$}]{a6}
		\Edge(a1)(a2)
		\Edge(a1)(a5)
		\Edge(a2)(a4)
		\Edge(a2)(a3)
		\Edge(a3)(a5)
		\Edge(a4)(a5)
		\Edge(a6)(a5)
		\Edge(a6)(a2);
		\draw[-{Latex}, ultra thick] (4,0) -- (8,0);
	\end{tikzpicture}
	\caption{Construction of $\mathfrak{L}_k$ from a $2$-regular graph on $4$ vertices for $k=2$.\label{fig:construction_L_from_C4}}
\end{figure}
Evidently, the graph $\mathfrak{L}_k$ captures all properties of $L$, i.e, we can reconstruct $L$ from $\mathfrak{L}_k$ up to isomorphisms. To this end, recall that each vertex can be interpreted as a vertex induced sub-graph of $L$ and all edges incident to a specific vertex in $\mathfrak{L}_k$ corresponds to an edge which points outward of the induced sub-graph. Iterating through all vertices in $\mathfrak{L}_k$ we can reconstruct all neighborhood relationships in $L$ and, therefore, $L$ up to isomorphisms. 

%% file: chapters/kPG_general_prop.tex
We consider exclusively connected graphs $L$ throughout this work, and stay in the beginning in an unconstrained setting with respect to the degree sequence of $L$. Later on we will reduce our scope to regular graphs to obtain deeper quantitative insights. It turns out that $L$ being regular is not a strong assumption when it comes to many applications for which the graph $\mathfrak{L}_k$ was initially conceptualized. Nonetheless, some important ones may be proven nonetheless and we present them in what follows.
\subsection{Topological properties of $\mathfrak{L}_k$}
 We start with the connectivity of $\mathfrak{L}_k$ which it inherits from $L$.
\begin{theorem}\label{thm:L_frak_k_connected}
	Let $L=(V,E)$ be any connected simple graph, $\bar{n}:=|V|$ and $k\in\{1,\hdots,\bar{n}-1\}$. Then, the associated graph  $\mathfrak{L}_k=(\mathfrak{V}_k,\mathfrak{E}_k)$ is connected.
\end{theorem}
Having shown that $\mathfrak{L}_k$ is connected if and only if $L$ is connected, we move on to bipartite graphs $L$ which do render $\mathfrak{L}_k$ bipartite. We show the response to this question in Theorem \ref{thm:L_bipartite_eq_L_k_bipartite}
\begin{theorem}\label{thm:L_bipartite_eq_L_k_bipartite}
	The graph $L$ is bipartite if and only if $\mathfrak{L}_k$ is bipartite.	
\end{theorem}
Next, we find that there is a natural link between $\mathfrak{L}_k$ and $\mathfrak{L}_k^c$ and the Johnson graph the natural overarching structure for the analysis of the graph $\mathfrak{L}_k$. In fact, the symmetric difference, which we used to define edges in $\mathfrak{L}_k$, allows for an even deeper identification, now for the graphs $\mathfrak{L}_k$ and $\mathfrak{L}_{\bar{n}-k}$ being practically identical as shall be shown in the following Theorem \ref{thm:quotient_graph_isomorphism}.
\begin{theorem}\label{thm:quotient_graph_isomorphism}
	Let $L$ be a simple connected graph on $\bar{n}$ vertices and $k\in\{1,\hdots,\bar{n}-1\}$. Denote by $\mathfrak{L}_k^c$ the $k$-particle graph induced by $L^c$. Then, the graphs $\mathfrak{L}_k^c$ and $(\mathfrak{L}_k)^c$ satisfy
	\begin{equation}
		\mathfrak{L}_k^c\stackrel{\sim}{=}(\mathfrak{L}_k)^c
	\end{equation}
	where the complement on the right hand side is taken with respect to the Johnson graph $J(\bar{n},k)$. Furthermore, the graph $\mathfrak{L}_k$ satisfies $\mathfrak{L}_k\stackrel{\sim}{=} \mathfrak{L}_{\bar{n}-k}$.
\end{theorem}
These observations on the topological properties of the $k$ particle graph form the basis for the following quantitative results, which we obtain via combinatorial considerations. 
\subsection{Geometric structure of $\mathfrak{L}_k$}
The $k$ particle graph $\mathfrak{L}_k$ carries all symmetries which are also present in $L$ and $L$ is the defining object for the local and global structure of $\mathfrak{L}_k$. We can capture the implied sub-graph structures in $\mathfrak{L}_k$ combinatorically, in particular, when considering cliques in $L$. 
\begin{theorem}
	Let $L$ a simple connected graph on $\bar{n}$ vertices, $k\in\{1,\hdots,\bar{n}-1\}$ and $c\in\mathbb{N}$. Then $L$ contains a clique of size $c$ if and only if $\mathfrak{L}_k$ contains a Johnson graph $\mathcal{J}(c,k')$ for all
	\begin{equation*}
		k'\in\{\max\{1,k-(\bar{n}-c)\},\hdots,\min\{k,c-1\}\}.
	\end{equation*}
	Additionally, with $\Delta_c^L$ denoting the number of cliques of size $c$ in $L$, the number $\iota_c^{k'}$ of Johnson graphs $\mathcal{J}(c,k')$ in $\mathfrak{L}_k$ is given by
	\begin{equation}
		\iota_c^{k'} = \binom{\bar{n}-c}{k-k'}\Delta_c^L.
	\end{equation}
\end{theorem}
Using the convention that $\binom{\bar{n}}{z}=0$ for all $z\in\mathbb{Z}^-$ we obtain the following result on the clique structure of $\mathfrak{L}_k$.
\begin{theorem}\label{thm:cliques_in_Lk}
	Let $L$ a simple connected graph on $\bar{n}$ vertices, $k\in\{1,\hdots,\bar{n}-1\}$ and $c\in\mathbb{N}$. Denote by $\Delta_c^L$ the number of cliques of size $c$ in $L$, the number $\Delta_c^{\mathfrak{L}_k}$ of cliques of size $c$ in $\mathfrak{L}_k$ is given by
	\begin{equation}
		\Delta_c^{\mathfrak{L}_k} = \left(\binom{\bar{n}-c}{k-1}+\binom{\bar{n}-c}{k-(c-1)}\right)\Delta_c^L.
	\end{equation}
\end{theorem}
The identity implies directly, that there are no large cliques in $\mathfrak{L}_k$. All clique sizes in $L$ imply at most one clique size in $\mathfrak{L}_k$. In the special case of a regular graph $L$, further claims can be made.
\begin{corollary}
	If $L$ is a $\bar{d}$ regular graph, but not complete, which contains a clique of size $c$, then $\mathfrak{L}_k$ contains a clique of size $c$. 
\end{corollary}
\begin{proof}
	By Theorem \ref{thm:cliques_in_Lk} the number of cliques of size $c$ in $\mathfrak{L}_k$ is given by
	\begin{equation*}
		\Delta_c^{\mathfrak{L}_k} = \left(\binom{\bar{n}-c}{k-1}+\binom{\bar{n}-c}{k-(c-1)}\right)\Delta_c^L.
	\end{equation*}
	This term is positive if and only if $\binom{\bar{n}-c}{k-1}+\binom{\bar{n}-c}{k-(c-1)}>0$ which amounts to $k-(c-1)>0$ or $k-1<\bar{n}-c$. Assume that $k-(c-1)<0$ and $k-1>\bar{n}-c$. Then, we conclude $\frac{\bar{n}}{2}+1 < c$. Therefore, the clique has size greater then half of the vertex set but $L$ is regular and not complete such that we obtain a contradiction.
\end{proof}
Furthermore, we can characterize the diameter of the $\mathfrak{L}_k$ as a function of the outer boundary of $k$ sub-graphs.
\begin{theorem}\label{thm:diameter_of_L_k}
	Let $L$ be a simple connected graph on $\bar{n}$ vertices with $\mathrm{diam}(L)=2$ and $k\in\left\{1,\hdots,\left\lfloor\frac{\bar{n}}{2}\right\rfloor\right\}$. Denote by $L_{\mathfrak{v}_{\ast}}$ the vertex induced sub-graph of $L$ on $k$ vertices $\mathfrak{v}_{\ast}\in\mathfrak{L}_k$ with the smallest outer boundary, i.e.,
	\begin{equation*}
		\mathcal{O}_b(L_{\mathfrak{v}_{\ast}}):=\left|\left(\bigcup_{v\in \mathfrak{v}_{\ast}} N_v\right)\backslash \mathfrak{v}_{\ast} \right|=\min_{\mathfrak{w}\in\mathfrak{V}_k} \left|\left(\bigcup_{w\in \mathfrak{w}} N_w\right)\backslash \mathfrak{w} \right|.
	\end{equation*}
	Define $\Delta_2^{(k)}(L):=\min\{\bar{n}-\mathcal{O}_b(L_{\mathfrak{v}_{\ast}})-k,k\}$. Then, the diameter of $\mathfrak{L}_k$ satisfies
	\begin{equation}
		\mathrm{diam}(\mathfrak{L}_k)= 2k-|k-\Delta_2(L)|=k+\Delta_2(L).
	\end{equation}
\end{theorem}
\begin{proof}
	Since the diameter of the graph is given by a $\max$-$\min$ problem over the paths in the graph, we have to consider paths in $\mathfrak{L}_k$. In fact, this can be interpreted as moving particle systems of size $k$ on the underlying $L$. Under the assumption $\mathrm{diam}(L)=2$ we can define for any $\mathfrak{v}\in\mathfrak{V}_k$ the outer boundary of $\mathfrak{v}$ as a sub-graph of $L$ by
	\begin{equation*}
		\mathcal{O}_b(L_{\mathfrak{v}}):=\left|\left(\bigcup_{v\in \mathfrak{v}} N_v\right)\backslash \mathfrak{v}\right|.
	\end{equation*}
	These are the vertices in $L$ which are not in $\mathfrak{v}$ and for which there is a $v\in\mathfrak{v}$ such that they are neighbors. Note that possibly the set $\{w\in\mathcal{O}_b(L_{\mathfrak{v}})|\exists v\in\mathfrak{v}:d_L(v,w)=2\}$ is not empty but due to the assumption $\mathrm{diam}(L)=2$ for all $v,w\in V$ we have $d_L(v,w)\leq 2$. Consider now the set $V\backslash \left(\mathcal{A}_b(L_{\mathfrak{v}})\sqcup \mathfrak{v}\right)$ which corresponds to the set of vertices which are at distance $2$ of all $v\in\mathfrak{v}$. We write
	\begin{equation*}
		\mathcal{A}(L_{\mathfrak{v}}) = \left|V\backslash \left(\mathcal{O}_b(L_{\mathfrak{v}})\sqcup \mathfrak{v}\right)\right|=\bar{n}-\left|\mathcal{O}_b(L_{\mathfrak{v}})\right|-k.
	\end{equation*}
	Then, exactly $\mathcal{A}(L_{\mathfrak{v}})$ particles in $\mathfrak{v}$ have to be displaced at most by $2$ to attain any other configuration $\mathfrak{w}$, even at a minimal distance. Consequently, the maximal minimal distance from $\mathfrak{v}$ to any $\mathfrak{w}$ is given by
	\begin{equation*}
		\max_{\mathfrak{w}} d(\mathfrak{v},\mathfrak{w}) = 2\max\{0,\min\{k,\mathcal{A}(L_{\mathfrak{v}})\}\} + \max\{0,\min\{\bar{n}-k,k-\mathcal{A}(L_{\mathfrak{v}})\}\}.	
	\end{equation*}	  
	Note that $\mathcal{A}(L_{\mathfrak{v}})$ is maximized by $\mathfrak{v}_{\ast}$ as defined in the assumptions of the theorem and since $\max_{\mathfrak{w}} d(\mathfrak{v},\mathfrak{w})$ is monotonous in $\mathcal{A}(L_{\mathfrak{v}})$ and $\min\{k,\mathcal{A}(L_{\mathfrak{v}_{\ast}})\}=\Delta_2^{(k)}(L)$ we conclude
	\begin{equation*}
		\mathrm{diam}(\mathfrak{L}_k)=\max_{\mathfrak{v},\mathfrak{w}} d(\mathfrak{v},\mathfrak{w}) = 2\max\{0,\min\{k,\mathcal{A}(L_{\mathfrak{v}_{\ast}})\}\} + \max\{0,\min\{\bar{n}-k,k-\mathcal{A}(L_{\mathfrak{v}_{\ast}})\}\}.
	\end{equation*}
	Under the assumption $k\in\left\{1,\hdots,\left\lfloor\frac{\bar{n}}{2}\right\rfloor\right\}$ we find $\bar{n}-k\geq \left\lfloor\frac{\bar{n}}{2}\right\rfloor \geq k\geq k-\mathcal{A}(L_{\mathfrak{v}})$	and using $\Delta_2^{(k)}(L)\geq 0$, we obtain
	\begin{align*}
		\mathrm{diam}(\mathfrak{L}_k) &= 2\min\{k,\Delta_2^{(k)}(L)\} + \max\{0,k-\mathcal{A}(L_{\mathfrak{v}_{\ast}})\}\\
		&= \min\{2k,2\Delta_2^{(k)}(L)\} + k- \Delta_2^{(k)}(L)\\
		&= k + \Delta_2^{(k)}(L) -|k-\Delta_2^{(k)}(L)| + k -\Delta_2^{(k)}(L) = 2k-|k-\Delta_2^{(k)}(L)| 
	\end{align*}
	and finally, since $\Delta_2^{(k)}(L)\leq k$ we conclude $\mathrm{diam}(\mathfrak{L}_k)= k + \Delta_2^{(k)}(L)$.
\end{proof}
We see from the equation $\mathrm{diam}(\mathfrak{L}_k)= 2k-|k-\Delta_2(L)|$ that the crude upper of the diameter of $\mathfrak{L}_k$ by $\mathrm{diam}(L)\,k$ is only correct if $\Delta_2(L)=k$, which implies $\bar{n}-\mathcal{O}_b(L_{\mathfrak{v}_{\ast}})-k\geq k$ and, therefore, also
\begin{equation*}
	\bar{n}-2k\geq \mathcal{O}_b(L_{\mathfrak{v}_{\ast}}). 
\end{equation*} 
Since $\mathcal{O}_b(L_{\mathfrak{v}_{\ast}})$ describes the smallest outer boundary of any vertex induced sub-graph on $k$ vertices this becomes a more and more restrictive condition on the geometry of $L$ as $k$ tends to $\left\lfloor\frac{\bar{n}}{2}\right\rfloor$. This leads to the following realization.
\begin{corollary}\label{cor:star_graph_Lk_diameter}
	Let $L$ be a star graph with $\bar{n}$ beams and $k\leq \frac{\bar{n}-1}{2}$ . Then, the diameter of $\mathfrak{L}_k$ satisfies $\mathrm{diam}(\mathfrak{L}_k)= 2k$. 
\end{corollary}
Therefore, due to Corollary \ref{cor:star_graph_Lk_diameter} the diameter $2$ graph $L$ which implies a maximal $k$ for which the particle has diameter $2k$ is the star graph, which has a very forgiving structure and which has many analytically describable properties. It gives also an intuition on how to prove further properties of $\mathfrak{L}_k$. We find explicitly the vertex connectivity of $\mathfrak{L}_k$.
\begin{theorem}\label{thm:L_k_vertex_connectivity}
	Let $L$ be a $\bar{d}$-regular graph on $\bar{n}$ vertices and let $k\in\{1,\hdots,\bar{n}-1\}$. Then, $\kappa(\mathfrak{L}_k) = \min_{\mathfrak{v}\in\mathfrak{V}_k}\mathrm{deg}_k(\mathfrak{v})$ and a corresponding vertex cut is the neighborhood of $\mathfrak{v}_{\ast}\in\mathfrak{V}_k$ with $\mathrm{deg}_k(\mathfrak{v}_{\ast}) = \min_{\mathfrak{v}\in\mathfrak{V}_k}\mathrm{deg}_k(\mathfrak{v})$. 
\end{theorem}
We illustrate first the proof of Theorem \ref{thm:L_k_vertex_connectivity}. Assume we have any $\bar{d}$-regular graph and $k=2$ particles. The proof is based on the idea of constructing marked path in $L$ and applying Menger's theorem to the ensemble of those paths. Fix to this end two configurations $\mathfrak{v}=\{v_1,v_2\}$ and $\mathfrak{w}=\{w_1,w_2\}$ which are two vertices in $\mathfrak{L}_2$. Now consider shortest paths $\phi_{v_1,w_1},\phi_{v_2,w_2}$ from $v_1$ to $w_1$ and from $v_2$ to $w_2$, respectively. We can obtain $\bar{d}-1$ distinct pairs of marked paths if either $v_1$ and $w_1$ or $v_2$ and $w_2$ are neighbors by the following construction. Assume without loss of generality that $v_1$ and $v_2$ are neighbors.
\begin{enumerate}
	\item Move the particle at $v_2$ to one of its neighbors $\tilde{v}_2$ which are not $v_1$.
	\item Move the particle at $v_1$ along $\phi_{v_1,w_1}$ to $w_1$.
	\item Move the particle at $w_1$ to any neighbor $\tilde{w}_1$ which is not $w_2$. 
	\item Move the particle at $\tilde{v}_2$ through $v_2$ along $\phi_{v_2,w_2}$ to $w_2$.
	\item Move the particle at $\tilde{w}_1$ to $w_1$. 
\end{enumerate}
Since there are $\bar{d}-1$ neighbors of $v_2$ and at least $\bar{d}-1$ neighbors of $w_1$ which allow this construction, we obtain $\bar{d}-1$ distinct pairs of marked paths. By symmetry of $v_1$ and $v_2$, we can construct exactly $2(\bar{d}-1)$ distinct pairs of marked graphs which correspond to $2(\bar{d}-1)$ paths in $\mathfrak{L}_2$. Since the minimal degree in $\mathfrak{L}_2$ is equal to $2(\bar{d}-1)$ we obtain by Menger's theorem the result given in Theorem \ref{thm:L_k_vertex_connectivity} for $k=2$ and by construction that the vertex cut is the neighborhood of the vertex in $\mathfrak{L}_2$ with minimal degree. The main challenge of the proof for arbitrary $k$ will be to rule out blockages between particles when they move along the shortest paths $\phi$ between their origin and their destination.\par
For the general case, the first point in the construction becomes: Move all particles $v_j$ but the particle at $v_i$ to one of their neighbors $\tilde{v}_j$ which are not $v_i$. The particle $v_i$ will then be moved as described before, then a second one goes to its original position, then moves and so on.
\begin{proof}
	Using the same construction as described before but with any number $k$ of particles, consider the initial configuration $\mathfrak{v}=\{v_1,\hdots,v_k\}$ and the finial configuration $\mathfrak{w}=\{w_1,\hdots,w_k\}$ and assume that $l$ particles $\{v_{i_1},\hdots,v_{i_l}\}$ have already been moved to neighbors $\{w_{i_1}^{(\mathfrak{n})},\hdots,w_{i_l}^{(\mathfrak{n})}\}$ of corresponding destinations $\{w_{i_1},\hdots,w_{i_l}\}$. Moreover, assume that $\{v_{i_{l+1}},\hdots,v_{i_{k-l-1}}\}$ have been moved to respective neighbors $\{v_{i_{l+1}}^{(\mathfrak{n})},\hdots,v_{i_{k-l-1}}^{(\mathfrak{n})}\}$. Finally, we want to move $v_{i_k}$ to $w_{i_k}$ along the shortest path $\phi_{v_{i_k},w_{i_k}}$. If no occupied vertices lie in $\phi_{v_{i_k},w_{i_k}}$, then nothing else has to be said. If there is an index $s$ such that $\phi_{v_{i_k},w_{i_k}}(s)$ is occupied, move $v_{i_k}$ to $\phi_{v_{i_k},w_{i_k}}(s-1)$, then move the particle in $\phi_{v_{i_k},w_{i_k}}(s)$ to $\phi_{v_{i_k},w_{i_k}}(s+1)$, then the particle which originates from $v_{i_k}$ to $\phi_{v_{i_k},w_{i_k}}(s)$ and, finally, the particle which is now in $\phi_{v_{i_k},w_{i_k}}(s+1)$ continues along $\phi_{v_{i_k},w_{i_k}}$. Since all other particles are in the neighborhood of initial or final positions, varying their positions in the respective neighborhoods leads to unique marked encounters along shortest paths. Consequently, the number of marks a single particle $v$, which moves to $w$, gives is $\min\{\bar{d}-\mathrm{deg}^{L_{\mathfrak{v}}}(v), \bar{d}-\mathrm{deg}^{L_{\mathfrak{w}}}(w)\}$. Iterating through all possible configurations of particles in neighborhoods of initial and final positions, we sum over all $v\in\mathfrak{v}$ and $w\in\mathfrak{w}$ and conclude that there are in total $\min\{\mathrm{deg}_k(\mathfrak{v}),\mathrm{deg}_k(\mathfrak{w})\}$ distinct ways to mark paths of particles in $\mathfrak{v}$ to sites in $\mathfrak{w}$. Consequently, we obtain that the number of sets of size $k$ of distinct marked paths and, therefore, the number of distinct paths in $\mathfrak{L}_k$ between $\mathfrak{v}$ and $\mathfrak{w}$ is given by $\min\{\mathrm{deg}_k(\mathfrak{v}),\mathrm{deg}_k(\mathfrak{w})\}$.\par
	Therefore, the minimal number of distinct paths between any two vertices $\mathfrak{v},\mathfrak{w}$ in $\mathfrak{L}_k$ is 
	\begin{equation}
		\min_{\mathfrak{v},\mathfrak{w}}\min\{\mathrm{deg}_k(\mathfrak{v}),\mathrm{deg}_k(\mathfrak{w})\}=\min_{\mathfrak{v}}\mathrm{deg}_k(\mathfrak{v}).
	\end{equation}
	By Menger's theorem, this yields $\kappa(\mathfrak{L}_k)=\min_{\mathfrak{v}}\mathrm{deg}_k(\mathfrak{v})$.
\end{proof}
Again, to illustrate the result in a more intuitive we come back to the example of the star graph and present a short proof in this special case, which gives a first taste of the general proof structure of Theorem \ref{thm:L_k_vertex_connectivity}.
\begin{corollary}
	Let $L$ be the star graph on $\bar{n}+1$ vertices and $\bar{n}$ beams. Let $k\in\{1,\hdots,\bar{n}\}$. The star graph $L$ has vertex connectivity $\kappa(L)=1$ and $\mathfrak{L}_k$ satisfies
	\begin{equation}
		\kappa(\mathfrak{L}_k) = \begin{cases}
							k,& k\leq \dfrac{\bar{n}+1}{2},\vspace{5pt}\\
							\bar{n}+1-k,& k\geq \dfrac{\bar{n}+1}{2}.
						 \end{cases}
	\end{equation}		 	
\end{corollary}
\begin{proof}
	We only consider the case $k\leq \frac{\bar{n}+1}{2}$. The second case follows by Proposition \ref{thm:quotient_graph_isomorphism}. We denote by $\alpha$ the center of the star and by $\beta_1,\hdots,\beta_n$ the vertices at the ends of the beams. Let $\mathfrak{v}\in\mathfrak{V}_k$. We have to consider two cases, namely $\alpha\in\mathfrak{v}$ and $\alpha\not\in\mathfrak{v}$. If $\alpha\not\in\mathfrak{v}$, i.e, $\mathfrak{v}\subseteq\{\beta_1,\hdots,\beta_n\}$, then any path $(\mathfrak{v},\mathfrak{u},\mathfrak{w})$ from $\mathfrak{v}$ to $\mathfrak{w}\subseteq\{\beta_1,\hdots,\beta_n\}$ of length $2$ has to include $\alpha$ in $\mathfrak{u}$. Consequently, any $\mathfrak{v}\subseteq\{\beta_1,\hdots,\beta_n\}$ has exactly $k$ neighbors and by the same argument any $\mathfrak{v}$ containing $\alpha$ has exactly $n+1-k$ neighbors. Furthermore, for $\mathfrak{v},\mathfrak{w}\in\mathfrak{V}_k$ with $\mathfrak{v}\neq\mathfrak{w}$ and $\alpha\in\mathfrak{v}$ as well as $\alpha\in\mathfrak{w}$ we have $\langle \mathfrak{v},\mathfrak{w} \rangle\not\in\mathfrak{E}_k$. Consequently, an edge $\langle\mathfrak{v},\mathfrak{w}\rangle\in\mathfrak{E}_k$ if and only if $\alpha \in\mathfrak{v}$ and $\alpha\not\in\mathfrak{w}$ or vice versa. Therefore, as long as $\mathfrak{v}$ is connected to at least one other vertex $\mathfrak{w}$ there is a path to any other $\mathfrak{v}'\in\mathfrak{V}_k$. We conclude that the vertex connectivity is given by the minimal degree which is $\min\{k,\bar{n}+1-k\}$. This yields the claim.
\end{proof}

We see that the $\mathfrak{L}_k$ carries over a lot of structure from $L$ which makes it richer but also more challenging from an analytical point of view. This is also true for the automorphism group of $\mathfrak{L}_k$, to which we turn now.
\begin{theorem}\label{thm:automorphism_group_L_k}
	Let $L$ be a simple connected graph $k\in\{1,\hdots,\bar{n}-1\}$. Then, the automorphism group $\mathrm{Aut}(L)$ of $L$ induces a sub-group of the automorphism group $\mathrm{Aut}(\mathfrak{L}_k)$ of $\mathfrak{L}_k$.
\end{theorem}
Indeed, the automorphism group of $\mathfrak{L}_k$ is larger then the sub-group induced by $\mathrm{Aut}(L)$. To this end, we consider identifiable sub-graphs $\mathfrak{v},\mathfrak{w}\in\mathfrak{V}_k$. Define for $\mathfrak{v},\mathfrak{w}\in\mathfrak{V}_k$ the equivalence relation $\mathfrak{v}\sim\mathfrak{w}$ if and only if $L_{\mathfrak{v},\mathfrak{v}^c}\stackrel{\sim}{=} L_{\mathfrak{w},\mathfrak{w}^c}$. For fixed $\mathfrak{v}_i$ define $[\mathfrak{v}_i]:=\{\mathfrak{u}\in\mathfrak{V}_k|\mathfrak{u}\sim\mathfrak{v}_i\}$ the equivalence class of $\mathfrak{v}_i$. We consider a fixed equivalence class $[\mathfrak{v}_i]$ and define for $\mathfrak{v},\mathfrak{w}\in [\mathfrak{v}_i]$ by $\Phi_{\mathfrak{v}}^{\mathfrak{w}}$ the isomorphism $L_{\mathfrak{v},\mathfrak{v}^c}\stackrel{\sim}{=} L_{\mathfrak{w},\mathfrak{w}^c}$. We can extend $\Phi_{\mathfrak{v}}^{\mathfrak{w}}$ to an automorphism of $\mathfrak{L}_k$ which is not induced by an automorphism on $L$. \\
First, consider the neighborhood of $\mathfrak{v}$. Since all vertices in the neighborhood of $\mathfrak{v}$ consist of subsets of $V$ of size $k$ which differ from $\mathfrak{v}$ by exactly one vertex in $L$ and $L_{\mathfrak{v},\mathfrak{v}^c}$ represents all possible transitions to neighbors of $\mathfrak{v}$, we exploit $L_{\mathfrak{v},\mathfrak{v}^c}\stackrel{\sim}{=} L_{\mathfrak{w},\mathfrak{w}^c}$ to obtain a mapping from the neighborhood of $\mathfrak{v}$ to the neighborhood of $\mathfrak{w}$ under which the equivalence classes with respect to $\sim$ are invariant. Iteratively, by this construction we obtain a map $\hat{\Phi}_{\mathfrak{v}}^{\mathfrak{w}}:\mathfrak{L}_k\to \mathfrak{L}_k$ which preserves the neighborhood property. Consequently, $\hat{\Phi}_{\mathfrak{v}}^{\mathfrak{w}}$ defines an automorphism of $\mathfrak{L}_k$. By changing $\mathfrak{v}$ and $\mathfrak{w}$ we obtain another such map since $\hat{\Phi}_{\mathfrak{v}}^{\mathfrak{w}}(\mathfrak{v})=\mathfrak{w}$ and, hence, $\hat{\Phi}_{\mathfrak{v}}^{\mathfrak{w}'}(\mathfrak{v})=\mathfrak{w}'$ for another $\mathfrak{w}'\in[\mathfrak{v}_i]$ which implies that $\mathrm{Aut}(\mathfrak{L}_k)$ is larger then the induced automorphism from $L$. 
\section{Proofs of general properties}
\begin{proof}[Proof of Theorem \ref{thm:L_frak_k_connected}]
	For this proof we use the concept of multi-sets, i.e., sets which allow the appearance of the same element multiple times. E.g. $\{1,1,2\}\neq \{1,2\}$. If a function acts on a multi-set in such a way that it changes one specific element which appears multiple times only one of them is altered. Assume for the rest of this proof that any set is a multi-set and any function mapping from sets to sets maps from multi-sets to multi-sets instead. In particular, any $\mathfrak{v}\in \mathfrak{V}_k$ will be considered as a multi-set.\par 
	Since $L$ is connected, there is for any pair $v,w\in V$ a self-avoiding path $\phi_{v,w}$ between $v$ and $w$. Let $\mathfrak{v},\mathfrak{w}\in \mathfrak{V}_k$. We are going to prove that there is path between $\mathfrak{v}$ and $\mathfrak{w}$. Define $\bar{\mathfrak{v}}:=\mathfrak{v}\backslash \mathfrak{w}$ and $\bar{\mathfrak{w}}$, analogously. Since $|\mathfrak{v}|=k=|\mathfrak{w}|$ also $|\bar{\mathfrak{v}}|=|\bar{\mathfrak{w}}|$. \\
	Fix $v_1\in \bar{\mathfrak{v}}$ and $w_1\in \bar{\mathfrak{w}}$. Then there exists a path $\phi_{v_1,w_1}^1$ in $L$. We want to construct iteratively a sequence of maps $(\Phi_i^1)_{i=1}^{|\phi_{v_1,w_1}|}$ by   
	\begin{equation}
		\Phi_1^1(\mathfrak{u}_0^1)=\left(\mathfrak{v}\backslash\{\phi_{v_1,w_1}^1(0)\}\right) \cup \{\phi_{v_1,w_1}^1(1)\}=:\mathfrak{u}_1^1
	\end{equation}
	with $u_0^1= \mathfrak{v}$.	Furthermore, $\mathfrak{u}_i^1 = \Phi_{i-1}^1(u_{i-1}^1)$. With the same procedure for $j=2,\hdots,|\bar{w}|$ and $\mathfrak{u}_0^j=\mathfrak{u}_{|\phi_{v_{j-1},w_{j-1}}|}^{j-1}$ we obtain a sequence of maps
	\begin{equation*}
		\Psi = \left(\Phi_1^{1},\hdots,\Phi_{|\phi_{v_{1},w_{1}}|}^{1},\Phi_{1}^{2},\hdots,\Phi_{|\phi_{v_{|\bar{\mathfrak{w}}|},w_{|\bar{\mathfrak{w}}|}}|}^{|\bar{\mathfrak{w}}|}\right)
	\end{equation*}
	which maps the $\mathfrak{v}$ to $\mathfrak{w}$ by $\left(\Phi_{|\phi_{v_{|\bar{\mathfrak{w}}|},w_{|\bar{\mathfrak{w}}|}}|}^{|\bar{\mathfrak{w}}|}\circ\hdots\circ\Phi_{1}^{2}\circ\Phi_{|\phi_{v_{1},w_{1}}|}^{1}\circ\hdots\circ\Phi_1^{1}\right)(\mathfrak{v})=\mathfrak{w}$. Elements might appear twice in the same $\mathfrak{u}_i^j$. Denote by $\tau$ the first entry in $\Psi$ such that $(\Psi_\tau\circ\hdots\circ\Psi_1)(\mathfrak{v})$ contains the same entry twice and by $\kappa$ the largest number such that $(\Psi_{\tau+\iota}\circ\hdots\circ\Psi_1)(\mathfrak{v})$ contains one entry twice for $\iota=0,\hdots,\kappa$. Transform the vector $(\Psi_{\iota'})_{\iota'=1}^{\tau+\kappa}$ into
	\begin{equation}
		(\Psi'_{\iota'})_{\iota'=1}^{\tau+\kappa}:=(\Psi_1,\hdots,\Psi_{\tau+\kappa},\Psi_{\tau+\kappa-1},\hdots,\Psi_{\tau}).
	\end{equation}		 
	Then $\left(\Psi'_{\tau+\iota}\circ\hdots\circ\Psi'_1\right)(\mathfrak{v})$ contains all elements only once for $\iota=0,\hdots,\kappa$. Iterate this procedure until $\Psi'=(\Psi'_{\iota'})_{\iota'=1}^{|\Psi|}$ such that for all $\iota'\in\{1,\hdots,|\Psi|\}$ we have $(\Psi'_{\iota'}\circ\hdots\circ\Psi'_1)(\mathfrak{v})\in \mathfrak{V}_k$ and for $\iota'\in\{1,\hdots,|\Psi|-1\}$
	\begin{equation}
		(\Psi'_{\iota'}\circ\hdots\circ\Psi'_1)(\mathfrak{v})\triangle(\Psi'_{\iota'+1}\circ\hdots\circ\Psi'_1)(\mathfrak{v}) = \{u_i^j,u_{i+1}^j\}
	\end{equation}		 
	with $\langle u_i^j,u_{i+1}^j\rangle\in E$ and $(\Psi'_{|\Psi|}\circ\hdots\circ\Psi'_1)(\mathfrak{v})=\mathfrak{w}$. Hence $\left((\Psi'_{\iota'}\circ\hdots\circ\Psi'_1)(\mathfrak{v})\right)_{\iota'=0}^{|\Psi|}$ defines a path from $\mathfrak{v}$ to $\mathfrak{w}$ in $\mathfrak{L}_k$.
\end{proof}
\begin{proposition}\label{prop:L_bipartite_L_k_bipartite}
	If $L$ is bipartite so is $\mathfrak{L}_k$.	
\end{proposition}
\begin{proof}
	We again use the notion of multi-sets. Assume $\mathfrak{L}_k$ is not bipartite and let $\phi=(\mathfrak{v}_1,\mathfrak{v}_2,\hdots,\mathfrak{v}_l,\mathfrak{v}_1)$ be a cycle of odd length, i.e, $l$ is an odd positive integer. Let $\mathfrak{v}_{l+1}=\mathfrak{v}_{1}$ and define the multi-set $E_{\phi}$ of edges in $L$ by
	\begin{equation*}
		E_{\phi} = \{\langle v,w \rangle\in E| \exists i\in\{1,\hdots,l\}, \mathfrak{v}_i\triangle\mathfrak{v}_{i+1}= \{v,w\}\}.
	\end{equation*}
	Since $\phi$ is a cycle we can construct, using the edges in $E_{\phi}$, for any $v\in \mathfrak{v}_1$ a cycle $\phi_v$ in $L$ such that the edges in all cycles $\phi_{v}$ combined correspond to $E_{\phi}$. Since $L$ is bipartite $|\phi_w|$ is even for every $w$ and $|E_{\phi}| = |\phi|$ is odd but
	\begin{equation*}
		|\phi| = |E_{\phi}| =\sum_{v\in\mathfrak{v}_1} |\phi_v| 
	\end{equation*}
	where the left hand side of the equation is odd and the right hand side is even which leads to a contradiction.		  		  
\end{proof}
Before we can continue with the proof of an equivalence for bipartite graphs we need a preliminary lemma, which adds additional information about a lower bound on the length of odd cycles in $L$ if $\mathfrak{L}_k$ is bipartite. This, in turn, yields a construction of a odd cycle in $\mathfrak{L}_k$ which will lead in Proposition \ref{prop:L_k_bipartite_L_bipartite} to the conclusion that $\mathfrak{L}_k$ bipartite implies also $L$ bipartite.
\begin{lemma}\label{lem:cycle_length_in_L_if_bipartite_L_k}
	If $\mathfrak{L}_k$ is bipartite then the shortest odd cycle in $L$ is longer than $\bar{n}-k+1$.
\end{lemma}
\begin{proof}
	Let $k\in\{1,\hdots,\bar{n}-1\}$ and assume there is an odd cycle $\phi=(v_1,v_2,\hdots,v_l,v_1)$ in $L$ with $|\phi|\leq \bar{n}-k+1$. Then choose a set $\mathfrak{v}\subset V\backslash(\{\phi_2,\hdots,\phi_l\})$ and $v_1\in\mathfrak{v}$. Note that $|V\backslash(\{\phi_2,\hdots,\phi_l\})|\geq k$ under the condition of the length of $\phi$. Define the cycle $\Phi = (\mathfrak{v}_1 =\mathfrak{v},\mathfrak{v}_2,\hdots,\mathfrak{v}_l,\mathfrak{v}_1)$ by $\mathfrak{v}_{i+1}=\mathfrak{v}_i\backslash\{v_i\}\cup\{v_{i+1}\}$. Then, $\Phi$ is an odd cycle which is a contradiction to the fact that $\mathfrak{L}_k$ is bipartite.    
\end{proof}
The goal is obviously to show that the lower bound shown in Lemma \ref{lem:cycle_length_in_L_if_bipartite_L_k} is rather conservative, the length of a non-existing walk being defined as infinity. Nonetheless, the intermediate step, deriving a lower bound on the length of odd cycles in $L$ if $\mathfrak{L}_k$ is bipartite, allows us the construction of a cycle in $\mathfrak{L}_k$ based on any odd cycle in $L$ with identical length. This yields the following result.
\begin{proposition}\label{prop:L_k_bipartite_L_bipartite}
	If $\mathfrak{L}_k$ is bipartite for some $k\in\{1,\hdots,\bar{n}-1\}$ then so is $L$.
\end{proposition}
\begin{proof}
	Also for this proof we use the concept of multi-sets and consider all $\mathfrak{v}$ as multi-sets. Consider $\mathfrak{L}_k$ as bipartite and assume $k\leq\dfrac{\bar{n}}{2}$ as well as $L$ not bipartite. Then there exists an odd cycle $\phi$ in $L$ and by assumption as well as Lemma \ref{lem:cycle_length_in_L_if_bipartite_L_k} we have $|\phi|\geq \bar{n}-k+1\geq \dfrac{\bar{n}}{2}+1\geq |\mathfrak{v}|$ for all $\mathfrak{v}\in\mathfrak{V}_k$. Consider $\mathfrak{v}\in\mathfrak{V}_k$ with $\phi_1\in\mathfrak{v}$. We define a sequence of maps $(\Psi_i)_{i=1}^{|\phi|}$ on multi-sets by 
	\begin{equation*}
		\Psi_i(\mathfrak{w}) =\begin{cases} (\mathfrak{w}\backslash\{\phi_i\})\cup\{\phi_{i+1}\},&\text{ if }\phi_i\in\mathfrak{w}\\
											\mathfrak{w},&\text{ otherwise.}												
							  \end{cases}
	\end{equation*}
	and define $\Psi = (\Psi_1,\hdots, \Psi_{|\phi|}\circ\hdots\circ\Psi_{1})$.	This way $(\Psi_{|\phi|}\circ\hdots\circ\Psi_{1})(\mathfrak{v})=\mathfrak{v}$ and $|(\Psi_{i}\circ\hdots\circ\Psi_{1})(\mathfrak{v})|=k$ for all $i=1,\hdots,|\phi|$. Along the lines of the proof of Proposition \ref{prop:L_frak_k_connected} denote by $\tau$ the first entry in $\Psi$ such that $(\Psi_\tau\circ\hdots\circ\Psi_1)(\mathfrak{v})$ contains the same entry twice and by $\kappa$ the largest number such that $(\Psi_{\tau+\iota}\circ\hdots\circ\Psi_1)(\mathfrak{v})$ contains one entry twice for $\iota=0,\hdots,\kappa$. Remark that both $\tau$ and $\kappa$ are well defined and finite due to $|\phi| \geq k$. Transform the vector $(\Psi_{\iota'})_{\iota'=1}^{\tau+\kappa}$ into
	\begin{equation}
		(\Psi'_{\iota'})_{\iota'=1}^{\tau+\kappa}:=(\Psi_1,\hdots,\Psi_{\tau+\kappa},\Psi_{\tau+\kappa-1},\hdots,\Psi_{\tau}).
	\end{equation}		 
	Then $\left(\Psi'_{\tau+\iota}\circ\hdots\circ\Psi'_1\right)(\mathfrak{v})$ contains all elements only once for $\iota=0,\hdots,\kappa$. Iterate this procedure until $\Psi'=(\Psi'_{\iota'}\circ\hdots\circ\Psi'_1)_{\iota'=1}^{|\phi|}$. Indeed, this leads to $(\Psi'_{|\phi|}\circ\hdots\circ\Psi'_{1})(\mathfrak{v})=\mathfrak{v}$ and $(\Psi'_{\iota'}\circ\hdots\circ\Psi'_{1})(\mathfrak{v})\in\mathfrak{V}_k$ for all $\iota'=1,\hdots,|\phi|$. Hence, the vector $((\Psi'_{\iota'}\circ\hdots\circ\Psi'_{1})(\mathfrak{v}))_{\iota'=1}^{|\phi|}$ defines a cycle of length $|\phi|$ in $\mathfrak{L}_k$. But $\phi$ is an odd cycle which leads to a contradiction. The claim follows since $\mathfrak{L}_k\stackrel{\sim}{=}\mathfrak{L}_{\bar{n}-k}$.
\end{proof}
To complete this first part on the global structure of $\mathfrak{L}_k$ for simple connected graphs $L$, we consider the complement of $L$ denoted by $L^c$ and the relationship between the induced graph $\mathfrak{L}_k^c$ and $(\mathfrak{L}_k)^c$ where the complement is taken with respect to the Johnson graph $J(\bar{n},k)$.

\begin{lemma}\label{lem:L_and_complement_L_link_Johnson_graph}
	Let $L=(V,E)$ be a simple connected graph on $\bar{n}:=|V|$ vertices and $k\in\{1,\hdots,\bar{n}-1\}$. Denote by $L_{\mathfrak{v}}=(\mathfrak{v}, E_{\mathfrak{v}})$ the vertex induced sub-graph of $\mathfrak{v}$ in $L$ and by $L_{\mathfrak{v}}^c=(\mathfrak{v}, E_{\mathfrak{v}}^c)$ the vertex induced sub-graph of $\mathfrak{v}$ in $L^c=(V,E_c)$. Denote by $\mathfrak{L}_k$ and $\mathfrak{L}_k^c$ the $k$-particle graphs of $L$ and $L^c$, respectively. Moreover, denote by $\mathrm{deg}_k(\mathfrak{v})$ the degree of $\mathfrak{v}$ in $\mathfrak{L}_k$ and by $\mathrm{deg}_k^c(\mathfrak{v})$ the degree of $\mathfrak{v}$ in $\mathfrak{L}_k^c$. Then,
	\begin{equation*}
		\mathrm{deg}_k(\mathfrak{v})+\mathrm{deg}_k^c(\mathfrak{v}) = k(\bar{n}-k).
	\end{equation*}
\end{lemma}
\begin{proof}
	For any subset $\mathfrak{v}\subset V$ with $|\mathfrak{v}|=k$ the size of the induced sub-graphs of $L$ and $L^c$ satisfy 
	\begin{equation*}
		|E_{\mathfrak{v}}|+|E_{\mathfrak{v}}^c| = \dfrac{k(k-1)}{2}.
	\end{equation*}
	Additionally, we derive
	\begin{align*}
		\mathrm{deg}_k(\mathfrak{v})+\mathrm{deg}_k^c(\mathfrak{v}) &= \sum_{v\in\mathfrak{v}} \mathrm{deg}(v) - \mathrm{deg}^{L_{\mathfrak{v}}}(v) + (\bar{n}-1-\mathrm{deg}(v)) - \mathrm{deg}^{L_{\mathfrak{v}}^c}(v)\\
		&= k(\bar{n}-1) - 2(|E_{\mathfrak{v}}|+|E_{\mathfrak{v}}^c|) = k(\bar{n}-1)- k(k-1) = k(\bar{n}-k).
	\end{align*}
\end{proof}
\begin{proof}[Proof of Theorem \ref{thm:quotient_graph_isomorphism}]
	Consider $(\mathfrak{L}_k)^c$. Then an edge $\langle \mathfrak{v},\mathfrak{w}\rangle$ in $J(\bar{n},k)$ is an edge in $(\mathfrak{L}_k)^c$ if and only if $\mathfrak{v}\triangle\mathfrak{w}=\{v,w\}$ and $\langle v,w\rangle\not\in E$. Consequently, $\langle \mathfrak{v},\mathfrak{w}\rangle$ is an edge in $(\mathfrak{L}_k)^c$ if and only if $\mathfrak{v}\triangle\mathfrak{w}=\{v,w\}$ and $\langle v,w\rangle\in E^c$. Therefore, $\langle \mathfrak{v},\mathfrak{w}\rangle$ is an edge in $\mathfrak{L}_k^c$. Since the vertex sets are identical and by Lemma \ref{lem:L_and_complement_L_link_Johnson_graph} each edge in $J(\bar{n},k)$ is either an edge in $\mathfrak{L}_k$ or $\mathfrak{L}_k^c$, we conclude $\mathfrak{L}_k^c\stackrel{\sim}{=}(\mathfrak{L}_k)^c$.\par
	We turn now to the second claim. By basic combinatorics of drawing without repetition $|\mathfrak{V}_k|=\binom{\bar{n}}{k}=|\mathfrak{V}_{\bar{n}-k}|$ holds true. Let $\mathfrak{v}\in \mathfrak{V}_k$ and define $\mathfrak{v}^c:=V\backslash\mathfrak{v}\in \mathfrak{V}_{\bar{n}-k}$.\par
	Define the map $\Phi_k^{\bar{n}-k}(\mathfrak{v})=\mathfrak{v}^c$ which is bijective due to the preceding observations. Let $\mathfrak{v},\mathfrak{w}\in\mathfrak{V}_k$ such that $\langle\mathfrak{v},\mathfrak{w}\rangle\in\mathfrak{E}_k$. Hence, $\mathfrak{v}\triangle\mathfrak{w}=\{v,w\}$ and $\langle v,w\rangle\in E$ and $\mathfrak{v}^c\triangle\mathfrak{w}^c=\{v,w\}$ and $\langle v,w\rangle\in E$. Therefore, also  $\Phi_k^{\bar{n}-k}(\mathfrak{v}) = \mathfrak{v}^c \sim \mathfrak{w}^c = \Phi_k^{\bar{n}-k}(\mathfrak{w})$. The map $\Phi_k^{\bar{n}-k}$, therefore, defines an isomorphism between $\mathfrak{L}_{k}$ and $\mathfrak{L}_{\bar{n}-k}$ and, consequently, $\mathfrak{L}_{k}\stackrel{\sim}{=} \mathfrak{L}_{\bar{n}-k}$.
\end{proof}
\begin{proof}[Proof of Theorem \ref{thm:automorphism_group_L_k}]
	We construct in what follows the explicit corresponding automorphism to any automorphism $\phi\in\mathrm{Aut}(L)$. Let $\phi\in \mathrm{Aut}(L)$ and define the map $\Phi$ on $\mathfrak{V}_k$ by $\Phi(\mathfrak{v})=\{\phi(v)|v\in\mathfrak{v}\}$. Then, if $\Phi(\mathfrak{v})\triangle \Phi(\mathfrak{w})=\{u,\bar{u}\}$ we have $\langle \Phi(\mathfrak{v}),\Phi(\mathfrak{w})\rangle\in\mathfrak{E}_k$ if and only if $\langle u,\bar{u}\rangle\in E$. Now assume that $\langle\mathfrak{v},\mathfrak{w}\rangle\in\mathfrak{E}_k$ and $\mathfrak{v}\triangle\mathfrak{w}=\{v,w\}$. Then, we obtain that $\Phi(\mathfrak{v})\triangle \Phi(\mathfrak{w})=\{u,\bar{u}\}=\{\phi(v),\phi(w)\}$ and since $\phi$ is an isomorphism on $L$ we arrive, moreover, at $\langle u,\bar{u}\rangle \in E$. Therefore, we can conclude $\langle \Phi(\mathfrak{v}),\Phi(\mathfrak{w})\rangle\in\mathfrak{E}_k$ which implies that any automorphism on $L$ defines an automorphism on $\mathfrak{L}_k$.
\end{proof}

%% file: chapters/kPG_regular_graph.tex
In what follows we focus on regular graphs and we assume, thus, for the reminder of this section, that $L$ is a $\bar{d}$-regular graph. This allows us to state explicit quantitative results as well as give the explicit correspondence of the relationship of a regular graph with a complete graph and the associated $\mathfrak{L}_k$ with $\mathcal{J}(\bar{n},k)$.
\subsection{Combinatorial properties of kPGs}
Note the Corollary \ref{cor:even_difference_of_degrees} yields that all vertices in $\mathfrak{L}_k$ have either even or odd degree. Hence, the size of vertices with even or odd degree is always the full vertex set size. We state for completeness its size which follows from standard combinatoric considerations as well as the size of the edge set, which demands a more involved analysis. 
\begin{proposition}\label{prop:size_edge_set_E_k}
	Consider the graph $\mathfrak{L}_k=(\mathfrak{V}_k, \mathfrak{E}_k)$ for $k\in\{1,\hdots,\bar{n}-1\}$. Then $|\mathfrak{V}_k|=\binom{\bar{n}}{k}$ and
	\begin{equation}\label{eq:size_of_edge_set_L_k_short}
		|\mathfrak{E}_k|=\dfrac{1}{2}\left(\bar{d}k\binom{\bar{n}}{k}-\bar{n}\bar{d}\binom{\bar{n}-2}{k-2}\right)=k(\bar{n}-k)\binom{\bar{n}}{k}\dfrac{\bar{d}}{2(\bar{n}-1)}.
	\end{equation}
\end{proposition}
\begin{proof}
	The first claim follows from drawing without replacement. For the edge set, by the formula for the degree of a $\mathfrak{v}\in\mathfrak{V}_k$, we obtain
	\begin{align*}
		2|\mathfrak{E}_k| &= \sum_{\mathfrak{v}\in\mathfrak{V}_k}\mathrm{deg}_k(\mathfrak{v})= \sum_{\mathfrak{v}\in\mathfrak{V}_k}k\bar{d}-\sum_{\mathfrak{v}\in\mathfrak{V}_k}\sum_{v\in\mathfrak{v}}\mathrm{deg}^{L_{\mathfrak{v}}}(v)\\
						  &= k\bar{d}\binom{\bar{n}}{k}-\sum_{v\in V}\sum_{\mathfrak{v}\in\mathfrak{V}_k}\mathrm{deg}^{L_{\mathfrak{v}}}(v)\mathbbm{1}_{v\in\mathfrak{v}}.
	\end{align*}
	Fixing a $v\in V$ we obtain for $s_v = \sum_{\mathfrak{v}\in\mathfrak{V}_k}\mathrm{deg}^{L_{\mathfrak{v}}}(v)\mathbbm{1}_{v\in\mathfrak{v}}$ that we have to redistribute the $k-1$ remaining particles among the $\bar{d}$ neighbors and the $\bar{n}-1-\bar{d}$ non-adjacent vertices. Assuming that $l$ particles are in the neighborhood of $\mathfrak{v}$ and one in $v$, there are $\binom{\bar{d}}{l}\binom{\bar{n}-1-\bar{d}}{k-1-l}$ ways to redistribute the particles in the aforementioned way and each contributes $l$ to to the sum $s_v$. Furthermore, $l$ ranges from $0$ to $\min\{\bar{d},k-1\}$ such that
	\begin{equation*}
		s_v = \sum_{l=0}^{\min\{\bar{d},k-1\}}\binom{\bar{d}}{l}\binom{\bar{n}-1-\bar{d}}{k-1-l}l 
	\end{equation*}		 
	independently of $v$. Consequently, we obtain the identity
	\begin{equation}\label{eq:size_of_edge_set_L_k}
			|\mathfrak{E}_k|=\dfrac{1}{2}\left(\bar{d}k\binom{\bar{n}}{k}-\bar{n}\sum_{l=1}^{\min\{k-1,\bar{d}\}}\binom{\bar{d}}{l}\binom{\bar{n}-1-\bar{d}}{k-1-l}l\right).
	\end{equation}
	For the second term, we remark that $\binom{\bar{d}}{l}=0$ for $l>\bar{d}$ and $\binom{\bar{n}-1-\bar{d}}{k-1-l}=0$ for $l>k-1$. Consequently, in the sums all summands with $l>\min\{k-1,\bar{d}\}$ are zero and, therefore,
	\begin{equation*}
		\sum_{l=1}^{k-1}\binom{\bar{d}}{l}\binom{\bar{n}-1-\bar{d}}{k-1-l}l=\sum_{l=1}^{\min\{k-1,\bar{d}\}}\binom{\bar{d}}{l}\binom{\bar{n}-1-\bar{d}}{k-1-l}l=\sum_{l=1}^{\bar{d}}\binom{\bar{d}}{l}\binom{\bar{n}-1-\bar{d}}{k-1-l}l.
	\end{equation*}
	From this we can continue with a preliminary observation that
	\begin{align*}
		\sum_{l=1}^{k-1}\binom{\bar{d}}{l}\binom{\bar{n}-1-\bar{d}}{k-1-l}l  &= \bar{d}\sum_{l=1}^{k-1}\binom{\bar{d}-1}{l-1}\binom{\bar{n}-1-\bar{d}}{k-1-l}.
	\end{align*}
	We omit $\bar{d}$ in what follows and prove only equality of the remaining terms. Note first that
	\begin{align*}
		\sum_{l=1}^{k-1}\binom{\bar{d}-1}{l-1}\binom{\bar{n}-1-\bar{d}}{k-1-l} & = \sum_{l=1}^{k-1}\binom{\bar{d}-1}{l-1}\binom{\bar{n}-2-(\bar{d}-1)}{k-2-(l-1)}\\
		& = \sum_{l=0}^{k-2}\binom{\bar{d}-1}{l}\binom{\bar{n}-2-(\bar{d}-1)}{k-2-l}
	\end{align*}
	such that we can apply the Zhu–Vandermonde identity to obtain
	\begin{align*}
		\sum_{l=1}^{k-1}\binom{\bar{d}-1}{l-1}\binom{\bar{n}-1-\bar{d}}{k-1-l} =\binom{\bar{n}-2}{k-2}
	\end{align*}
	which finishes the proof.
\end{proof}
Hence, we obtain an equality for the size of the vertex sets for $k$ and $\bar{n}-k$ as well as for the size of the edge sets in the two cases. Indeed, it is not obvious that the value given in equation \ref{eq:size_of_edge_set_L_k_short} is an integer. When checking this property one arrives at the following conclusion.
\begin{corollary}\label{cor:size_edge_set_E_k_comp_size_E}
	Let $L=(V,E)$ be a $\bar{d}$-regular graph. Consider the graph $\mathfrak{L}_k$ for $k\in\{1,\hdots,\bar{n}-1\}$. Then, 
	\begin{equation}\label{eq:size_of_edge_set_L_k_comp_L}
		|\mathfrak{E}_k|=\binom{\bar{n}-2}{k-1}|E|.
	\end{equation}
\end{corollary}
\begin{proof}
	We employ equation \ref{eq:size_of_edge_set_L_k_short} to derive the claim.
	\begin{align*}
		|\mathfrak{E}_k|&=k(\bar{n}-k)\binom{\bar{n}}{k}\dfrac{\bar{d}}{2(\bar{n}-1)}=\bar{n}\dfrac{(\bar{n}-2)!}{(k-1)!(\bar{n}-2-(k-1))!}\dfrac{\bar{d}}{2}\\
		&= \binom{\bar{n}-2}{k-1}\dfrac{\bar{n}\bar{d}}{2}=\binom{\bar{n}-2}{k-1}|E|.
	\end{align*}
\end{proof}
On the other hand comparing the result presented in equation \ref{eq:size_of_edge_set_L_k_short} with the size of the edge set $E_{\mathcal{J}}$ of a $\mathcal{J}(\bar{n},k)$ Johnson graph, we obtain that 
\begin{equation*}
	\dfrac{|\mathfrak{E}_k|}{|E_{\mathcal{J}}|}=\dfrac{\bar{d}}{\bar{n}-1}
\end{equation*} 
which is independent of $k$ and only dependent on the underlying graph $L$ through the prescribed degree $\bar{d}$. Based on this remark we can also make the following core observation for the correspondence of the relationship of a regular graph with a complete graph and the associated $\mathfrak{L}_k$ with $\mathcal{J}(\bar{n},k)$.
\begin{proposition}\label{prop:average_degree_quotient_upper_bound}
	Let $L$ be a $\bar{d}$-regular graph on $\bar{n}$ vertices and $k\in\{1,\hdots,\bar{n}-1\}$. Denote by $\mathrm{avg\;deg}(\mathfrak{L}_k)$ the average degree of $\mathfrak{L}_k$, i.e., 
	\begin{equation*}
		\mathrm{avg\;deg}(\mathfrak{L}_k):=\dfrac{1}{|\mathfrak{V}_k|}\sum_{\mathfrak{v}\in\mathfrak{V}_k}\mathrm{deg}_k(\mathfrak{v}).
	\end{equation*}
	The average degree satisfies
	\begin{equation}\label{eq:expression_quotient_average_degree}
		\dfrac{\mathrm{avg\;deg}(\mathfrak{L}_k)}{k(\bar{n}-k)}=\dfrac{\bar{d}}{\bar{n}-1}\leq 1.
	\end{equation}
\end{proposition} 
\begin{proof}
	We start by recalling that the average degree of any graph is given by the quotient of the size of its edge set times two and the size of its vertex set. Therefore,
	\begin{align*}
		\dfrac{\mathrm{avg\;deg}(\mathfrak{L}_k)}{k(\bar{n}-k)} &= \dfrac{2|\mathfrak{E}_k|}{|\mathfrak{V}_k|\,k(\bar{n}-k)}=\dfrac{\bar{d}}{\bar{n}-k} - \dfrac{\bar{d}}{\binom{\bar{n}-1}{k-1}}\dfrac{\binom{\bar{n}-2}{k-2}}{\bar{n}-k}\\
		&= \dfrac{\bar{d}}{\bar{n}-k}\left(1-\dfrac{k-1}{\bar{n}-1}\right) = \dfrac{\bar{d}}{\bar{n}-1}.
	\end{align*}
	The second claim follows from $\bar{d}\leq \bar{n}-1$ and equality if and only if $L$ is a complete graph.
\end{proof}
The average degree is not the only object defined by the degree sequence we are interested in. We turn to the properties of the degree sequence of $\mathfrak{L}_k$.
\subsection{The degree sequence of $\mathfrak{L}_k$}
Indeed, for regular graphs we obtain a range of structural results on $\mathfrak{L}_k$ which go beyond connectedness and focus on the local properties of the vertices. In particular the degree of a vertex plays a central role defining in later Sections the transition probabilities and stationary distribution of Markov chains on $\mathfrak{L}_k$ induced by a variety of exclusion processes.
\begin{proposition}\label{prop:equivalence_equality_degree_equality_degree_subgraph}
	Let $k\in\{1,\hdots,\bar{n}-1\}$ and $\langle\mathfrak{v},\mathfrak{w}\rangle\in\mathfrak{E}_k$ and write $\mathfrak{v}\triangle\mathfrak{w}=\{v,w\}$. Then $\mathrm{deg}_k(\mathfrak{v})=\mathrm{deg}_k(\mathfrak{w})$ if and only if $\mathrm{deg}^{L_{\mathfrak{v}}}(v)=\mathrm{deg}^{L_{\mathfrak{w}}}(w)$.\\
	Moreover, any $\mathfrak{v},\mathfrak{w}\in\mathfrak{V}_k$ with $\langle\mathfrak{v},\mathfrak{w}\rangle\in\mathfrak{E}_k$ and $\mathfrak{v}\triangle\mathfrak{w}=\{v,w\}$ satisfy
	\begin{equation}
		\mathrm{deg}^{L_{\mathfrak{v}}}(v) + \mathrm{deg}^{L_{V\backslash\mathfrak{w}}}(v) =\bar{d}-1.
	\end{equation}
\end{proposition}
\begin{proof}
	First of all, note that $\mathfrak{v}\backslash\{v\}=\mathfrak{w}\backslash\{w\}$ and, hence, $\mathrm{deg}_{k-1}(\mathfrak{v}\backslash\{v\})=\mathrm{deg}_{k-1}(\mathfrak{w}\backslash\{w\})$. Furthermore, removing $v$ from $\mathfrak{v}$ removes $\mathrm{deg}^{L_{\mathfrak{v}}}(v)$ edges from the induced subgraph $L_{\mathfrak{v}}$. Consequently, we obtain $\mathrm{deg}_k(\mathfrak{v})=\mathrm{deg}_{k-1}(\mathfrak{v}\backslash\{v\})+\bar{d} -2\,\mathrm{deg}^{L_{\mathfrak{v}}}(v)$. Equivalent claims are satisfied by $\mathfrak{w}$ and $w$. Hence, we can conclude
	\begin{align*}
		\mathrm{deg}_k(\mathfrak{v})-\mathrm{deg}_k(\mathfrak{w}) &= \mathrm{deg}_{k-1}(\mathfrak{v}\backslash\{v\}) - 2\,\mathrm{deg}^{L_{\mathfrak{v}}}(v) - (\mathrm{deg}_{k-1}(\mathfrak{w}\backslash\{w\}) - 2\,\mathrm{deg}^{L_{\mathfrak{w}}}(w))\\
		&= 2\,(\mathrm{deg}^{L_{\mathfrak{w}}}(w)-\mathrm{deg}^{L_{\mathfrak{v}}}(v))
	\end{align*}		 
	which is equivalent to the first claim.\\
	Secondly, note that $\mathfrak{v}\backslash\{v\}=\mathfrak{w}\backslash\{w\}$ and, hence, $\mathrm{deg}_{k-1}(\mathfrak{v}\backslash\{v\})=\mathrm{deg}_{k-1}(\mathfrak{w}\backslash\{w\})$. Again, removing $v$ from $\mathfrak{v}$ removes $\mathrm{deg}^{L_{\mathfrak{v}}}(v)$ edges from the induced subgraph $L_{\mathfrak{v}}$. Consequently, we obtain $\mathrm{deg}_k(\mathfrak{v})=\mathrm{deg}_{k-1}(\mathfrak{v}\backslash\{v\})+\bar{d} -2\,\mathrm{deg}^{L_{\mathfrak{v}}}(v)$. Equivalent claims are satisfied by $\mathfrak{w}$ and $w$. Hence, we can conclude
	\begin{align*}
		\mathrm{deg}_k(\mathfrak{v})-\mathrm{deg}_k(\mathfrak{w}) &= \mathrm{deg}_{k-1}(\mathfrak{v}\backslash\{v\}) - 2\,\mathrm{deg}^{L_{\mathfrak{v}}}(v) - (\mathrm{deg}_{k-1}(\mathfrak{w}\backslash\{w\}) - 2\,\mathrm{deg}^{L_{\mathfrak{w}}}(w))\\
		&= 2\,(\mathrm{deg}^{L_{\mathfrak{w}}}(w)-\mathrm{deg}^{L_{\mathfrak{v}}}(v))
	\end{align*}		 
	which is equivalent to the claim.
\end{proof}
Proposition \ref{prop:equivalence_equality_degree_equality_degree_subgraph} gives a perspective on the graph $\mathfrak{L}_k$ that in fact the $k$ vertex induced sub-graphs of $L$ are the defining objects. While their analysis is a classically difficult subject, see again for example \cite{Khuller2009OnFD}, \cite{Charalampos2015} and \cite{Feige2001}, we obtain nonetheless properties based on construction of $\mathfrak{L}_k$. 
\begin{corollary}\label{cor:even_difference_of_degrees}
	Let $k\in\{1,\hdots,\bar{n}-1\}$ and $\mathfrak{v},\mathfrak{w}\in\mathfrak{V}_k$. Then, $\mathrm{deg}_k(\mathfrak{v})-\mathrm{deg}_k(\mathfrak{w})$ is an even number.
\end{corollary}
\begin{proof}
	Consider first $\langle\mathfrak{v},\mathfrak{w}\rangle\in\mathfrak{E}_k$. Then, by the proof of Proposition \ref{prop:equivalence_equality_degree_equality_degree_subgraph} we have $\mathrm{deg}_k(\mathfrak{v})-\mathrm{deg}_k(\mathfrak{w})=2\,(\mathrm{deg}^{L_{\mathfrak{w}}}(w)-\mathrm{deg}^{L_{\mathfrak{v}}}(v))$. For arbitrary $\mathfrak{v},\mathfrak{w}\in\mathfrak{V}_k$ we can construct a path from $\mathfrak{v}$ to $\mathfrak{w}$ by Proposition \ref{prop:L_frak_k_connected} and for any segment the difference of degrees is even. Hence, by a bootstrap argument also $\mathrm{deg}_k(\mathfrak{v})-\mathrm{deg}_k(\mathfrak{w})$ is even.
\end{proof}
We investigate in what follows the link between the edges in vertex induced sub-graphs and the degree of a vertex in $\mathfrak{L}_k$.
 	\begin{proposition}\label{prop:degree_formula_in_tilde_H_k_r}
	Let $L$ be a $\bar{d}$-regular graph, $\mathfrak{v}\in\mathfrak{L}_k$ and denote by $L_{\mathfrak{v}}= (\mathfrak{v},E_{\mathfrak{v}})$ the vertex induced sub-graph of $L$. Then
	\begin{equation}
		\mathrm{deg}_k(\mathfrak{v})= k\cdot \bar{d} - 2|E_{\mathfrak{v}}|.
	\end{equation}
	Additionally, denote by $L_{\mathrm{min};k}=(V_{\mathrm{min};k}, E_{\mathrm{min};k})$ a least dense vertex induced sub-graph on $k$ vertices of $L$ and by $L_{\mathrm{max};k}=(V_{\mathrm{max};k}, E_{\mathrm{max};k})$ a densest vertex induced sub-graph on $k$ vertices. Then
	\begin{align*}
		\min_{\mathfrak{v} \in \mathfrak{V}_k}\mathrm{deg}_k(\mathfrak{v}) &= k\cdot\bar{d} - 2|E_{\mathrm{max};k}|,\\
		\max_{\mathfrak{v} \in \mathfrak{V}_k}\mathrm{deg}_k(\mathfrak{v}) &= k\cdot\bar{d} - 2|E_{\mathrm{min};k}|.
	\end{align*}
\end{proposition}
\begin{proof}
	For $\mathfrak{v}\in\mathfrak{V}_k$ consider the vertex induced sub-graph $L_{\mathfrak{v}}=(V_{\mathfrak{v}},E_{\mathfrak{v}})$ of $L$. Any $v\in V$ has $\bar{d}$ neighbors. Hence, the degree of $\mathfrak{v}$ in $\mathfrak{L}_k$ has the form $\mathrm{deg}^r_k(\mathfrak{v}) =k\cdot\bar{d} - m(\mathfrak{v})$ where $m(\mathfrak{v})$ is defined by the constrains given through the definition of $\mathfrak{E}_k$ and remains to be determined. Since $\mathfrak{v}\sim\mathfrak{w}$ if and only if $\mathfrak{v}\triangle\mathfrak{w}=\{v,w\}$ and $\langle v,w\rangle\in E$ any edge in $L$ between $v,v'\in \mathfrak{v}$ reduce the degree by two due to symmetry of the edge $\{v,v'\}$. Consequently, 
	\begin{equation*}
		\mathrm{deg}_k(\mathfrak{v})= k\cdot \bar{d} - 2|E_{\mathfrak{v}}|.
	\end{equation*}
	By the first result we can use the identity 
	\begin{equation}
		\mathrm{deg}_k(\mathfrak{v})= k\cdot \bar{d} - 2|E_{\mathfrak{v}}|
	\end{equation}
	for any $\mathfrak{v}\in\mathfrak{V}_k$. Therefore, 
	\begin{equation}
		\tilde{\delta}_{k;\ast}:=\min_{\mathfrak{v} \in \mathfrak{V}_k}\mathrm{deg}_k(\mathfrak{v}) = k\cdot \bar{d} - 2\max_{\mathfrak{v} \in \mathfrak{V}_k} |E_{\mathfrak{v}}|= k\cdot\bar{d} - 2|E_{\mathrm{max};k}|.
	\end{equation}
	The same argumentation is valid for $\tilde{\delta}^{k;\ast}$. 
\end{proof}
Having established the relation between the $k$ vertex induced sub-graphs and in particular the size of their edge sets with the degrees in $\mathfrak{L}_k$ we can then come back to the question about the degrees for varying $k$ and the fact that the degrees seem to be either all even or all odd. We show this property first before going on to the former. 
\begin{lemma}\label{lem:isom_of_bipartites_implies_same_degree}
Let $k\in\left\{1,\hdots,\bar{n}-1\right\}$ and let $\mathfrak{v}, \mathfrak{w} \in \mathfrak{V}_k$. Then, $L_{\mathfrak{v},\mathfrak{v}^c}\stackrel{\sim}{=} L_{\mathfrak{w},\mathfrak{w}^c}$ implies $\mathrm{deg}_k(\mathfrak{v}) = \mathrm{deg}_k(\mathfrak{w})$.
	Additionally, for any $\mathfrak{v}\in\mathfrak{V}_k$ its degree $\mathrm{deg}_k(\mathfrak{v})$ is even if and only if $k\bar{d}$ is even. Hence, if there is $\mathfrak{v}\in\mathfrak{V}_k$ such that $\mathrm{deg}_k(\mathfrak{v})$ is even it is true that for all $\mathfrak{w}\in\mathfrak{V}_k$ the number $\mathrm{deg}_k(\mathfrak{w})$ is even.
\end{lemma}
\begin{proof}
	The first claim follows immediately since $\mathrm{deg}_k(\mathfrak{v}) = |E_{\mathfrak{v},\mathfrak{v}^c}|$ and $|E_{\mathfrak{v},\mathfrak{v}^c}|=|E_{\mathfrak{w},\mathfrak{w}^c}|$.\par 
	Drawing from Proposition \ref{prop:degree_formula_in_tilde_H_k_r} we obtain that 
	\begin{equation*}
		\mathrm{deg}_k(\mathfrak{v}) + 2|E_{\mathfrak{v}}|= k\cdot \bar{d}
	\end{equation*}
	which yields the first claim. The second claim follows by Corollary \ref{cor:even_difference_of_degrees}.
\end{proof}
Indeed, the importance of the graph $L_{\mathfrak{v},\mathfrak{v}^c}=((\mathfrak{v},\mathfrak{v}^c), E_{\mathfrak{v},\mathfrak{v}^c}) $ does not only reduce to the fact that it may be a minimizer as discussed in Lemma \ref{lem:isom_of_bipartites_implies_same_degree} but it does in fact characterize both $\mathfrak{v}$ and $\mathfrak{v}^c$ at the same time. Unfortunately, obtaining further identities would amount to solving questions about sizes of sub-graphs of a $\bar{d}$-regular graph $L$. This is a well known problem in various fields but aside from approximations via algorithmic approaches it remains out of reach of being solved. Nonetheless, further qualitative claims can be made about $\mathfrak{L}_k$ which are in particular symmetry based observations.
\subsection{Symmetries $k$-$(\bar{n}-k)$ and under complements}
The first symmetry which is an immediate consequence of the previously established isomorphism between $\mathfrak{L}_k$ and $\mathfrak{L}_{\bar{n}-k}$ concerns the degree set which is an invariant under isomorphisms.
\begin{proposition}\label{prop:def_Dk_equality}
Let $k\in\{1,\hdots,\bar{n}-1\}$ and $D_k :=\{\mathrm{deg}_k^r(v)|v\in \mathfrak{V}_{k}\}$. Let $k,k'\in\left\{1,\hdots,\left\lfloor\dfrac{\bar{n}}{2}\right\rfloor\right\}$, $k'\leq k$. Then $|D_{k'}|\leq |D_k|$. If in turn $k,k'\in\left\{\left\lceil\dfrac{\bar{n}}{2}\right\rceil,\hdots,\bar{n}-1\right\}$, $k\leq k'$ then $|D_k| \geq |D_{k'}|$.
\end{proposition}
\begin{proof}
	Let $k\leq \left\lfloor\dfrac{\bar{n}}{2}\right\rfloor$. Then, for any $\hat{\mathfrak{v}}\in\mathfrak{V}_{k-1}$ we can define the set $\hat{\mathfrak{V}}_k:=\{\mathfrak{v}\in\mathfrak{V}_k|\mathfrak{v}\cap\hat{\mathfrak{v}}=\hat{\mathfrak{v}}\}$ and $|\hat{\mathfrak{V}}_k|= \bar{n}-(k-1)\geq 1$. For any $\mathfrak{v}\in \hat{\mathfrak{V}}_k$ the degree of $\mathfrak{v}$ is given by the proof of Proposition \ref{prop:equivalence_equality_degree_equality_degree_subgraph} by $\mathrm{deg}_k(\mathfrak{v})=\mathrm{deg}_{k-1}(\hat{\mathfrak{v}})+\bar{d} -2\,\mathrm{deg}^{L_{\mathfrak{v}}}(v)$ where $\{v\}:=\mathfrak{v}\backslash\hat{\mathfrak{v}}$. Hence, any degree $d\in D_{k-1}$ defines at least one degree $d'\in\  D_k$. Consequently, we obtain $|D_{k-1}|\leq |D_k|$ and by a bootstrap argument also $|D_{k'}|\leq |D_k|$ for any $k'\leq k$.\par 
	The second claim of the Lemma follows from Proposition \ref{prop:quotient_graph_isomorphism}.
\end{proof}
Note that the inequalities in Proposition \ref{prop:def_Dk_equality} do not have any implications on the inclusion of the degree sets. Consider to this end the graphs depicted in Figure \ref{fig:comparison_cycle_graph_8} and Figure \ref{fig:comparison_regular_3_graph}. While the underlying graph, the cycle graph on $8$ vertices. 
\begin{figure}[!htb]
	\begin{subfigure}{0.32\textwidth}
		\centering
		\includegraphics[scale=0.1]{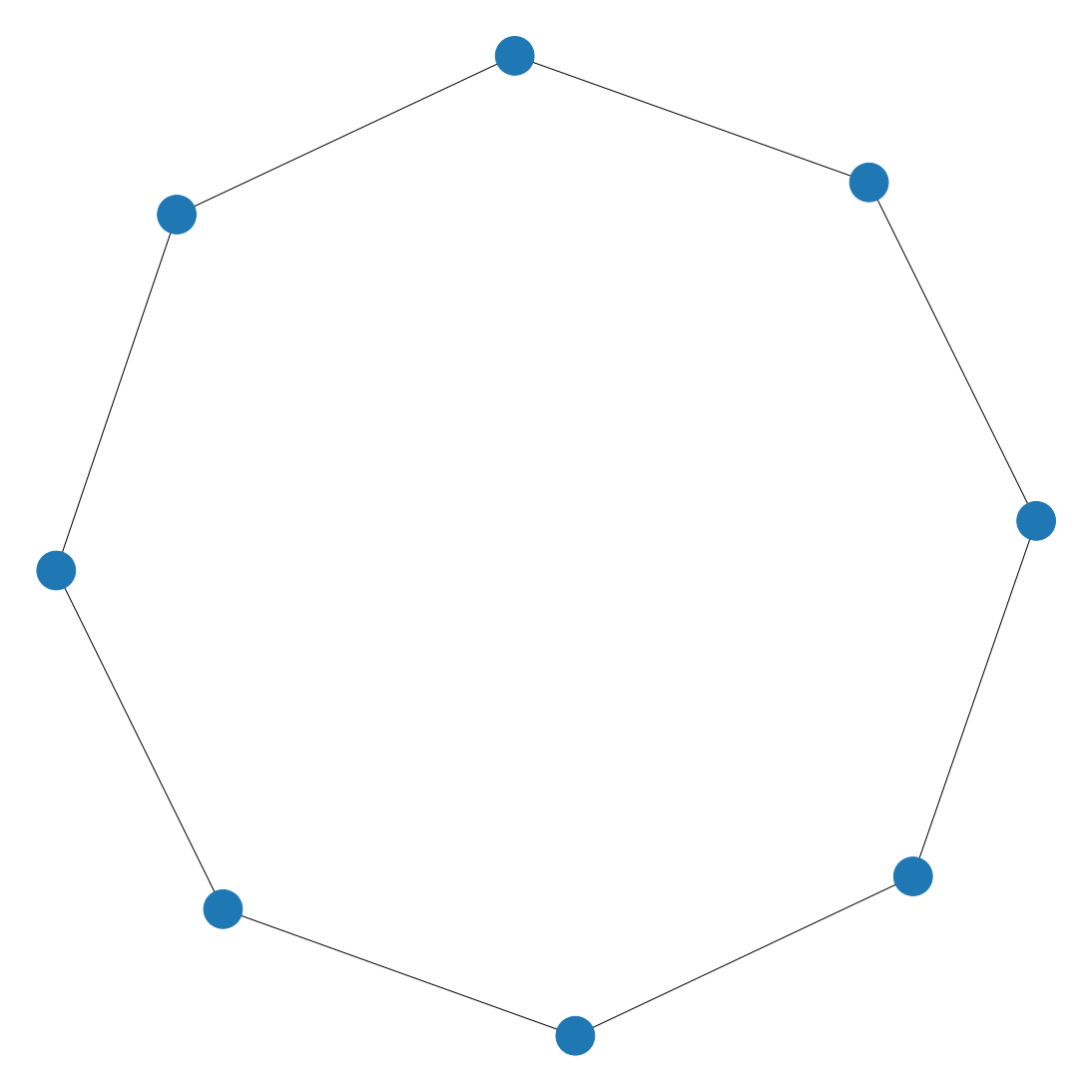}
	\end{subfigure}\hfill
	\begin{subfigure}{0.32\textwidth}
		\centering
		\includegraphics[scale=0.1]{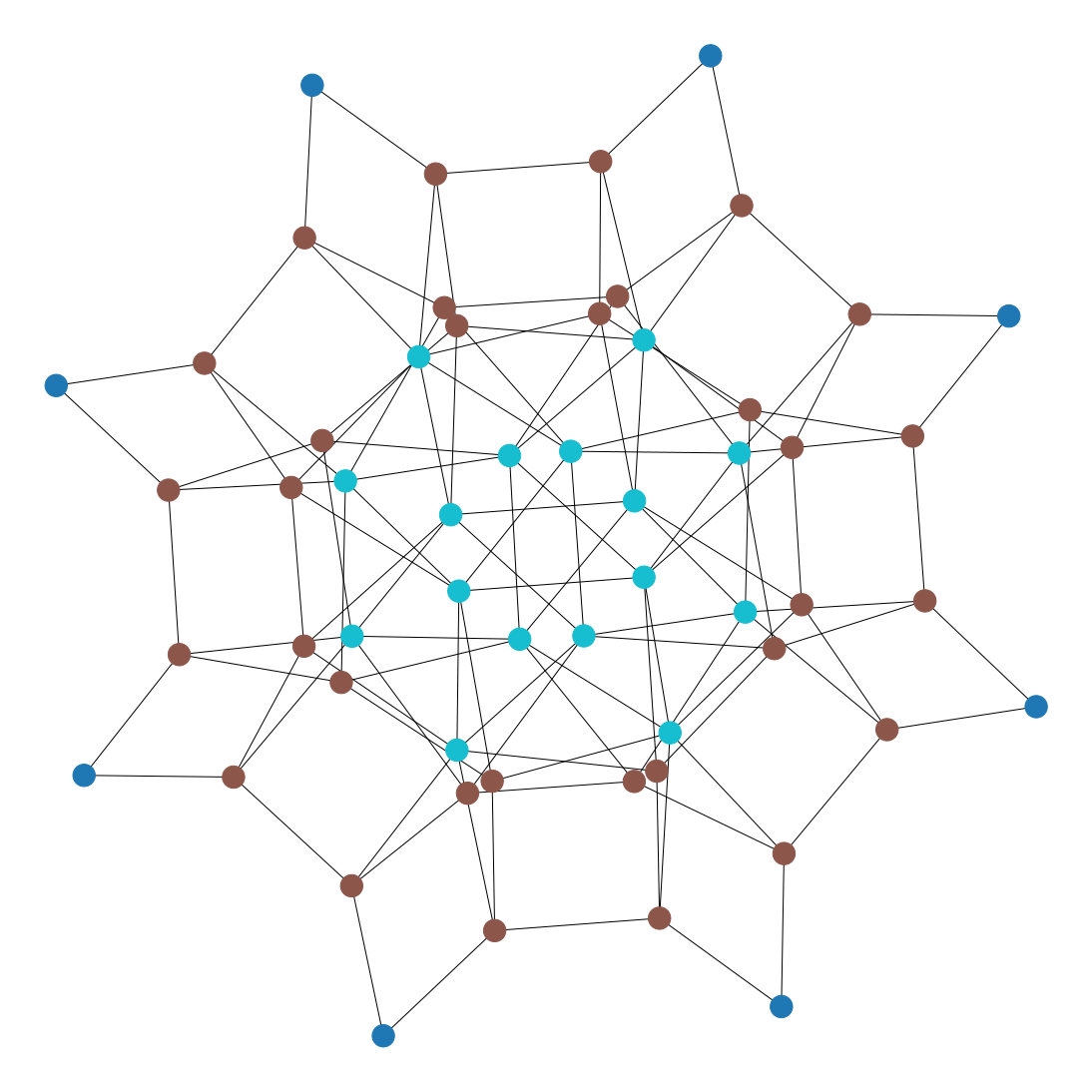}
	\end{subfigure}\hfill
	\begin{subfigure}{0.32\textwidth}
	 	\centering
		\includegraphics[scale=0.1]{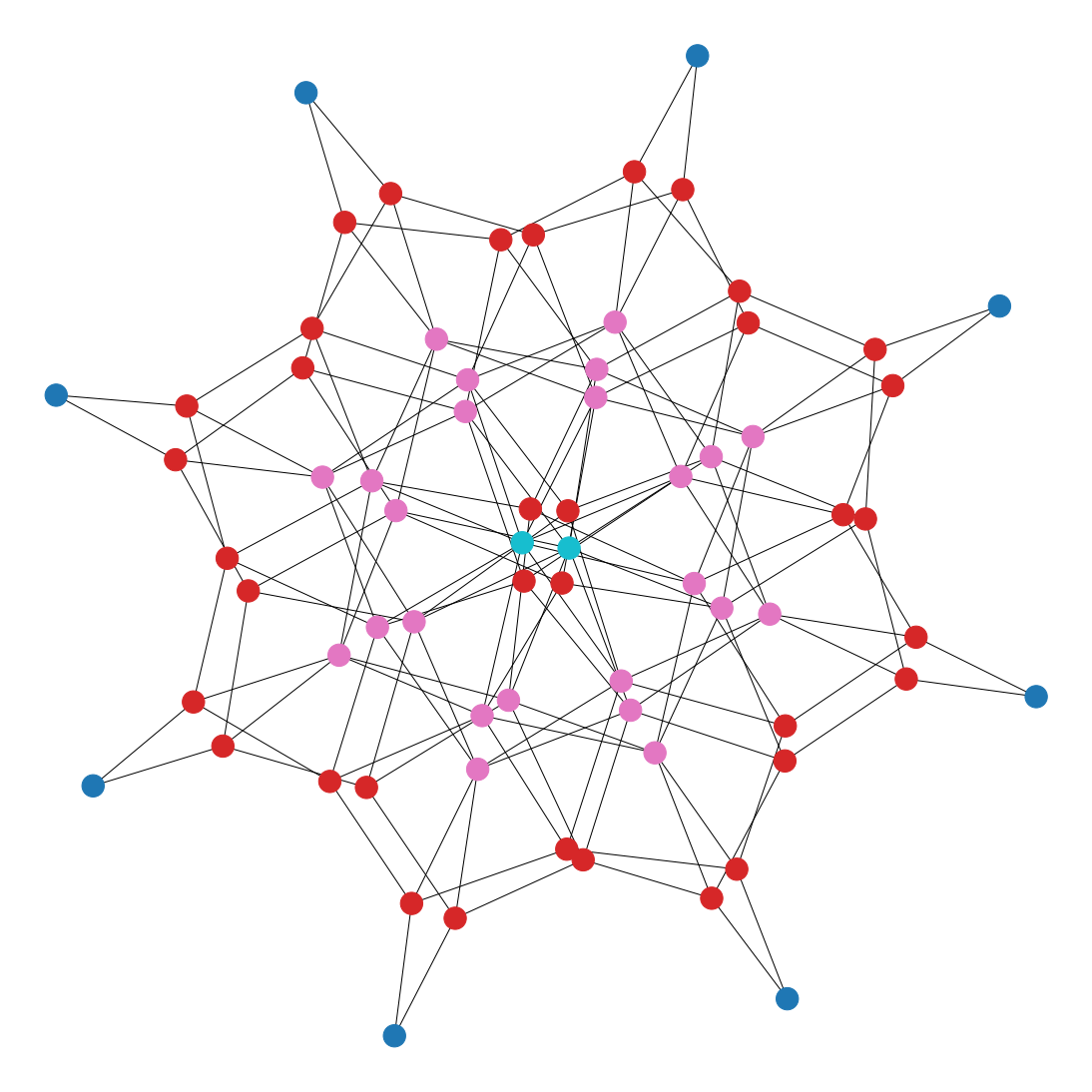}
\end{subfigure}
	\caption{From left to right the underlying cycle graph, the associated graphs $\mathfrak{L}_3$ and $\mathfrak{L}_4$. Indeed, due to the configurations of the vertex induced subgraphs, inclusion $D_3\subset D_4$ is given.\label{fig:comparison_cycle_graph_8}}
\end{figure}
Then, the degree sets for $k=3$ and $k=4$ satisfy
\begin{equation*}
	D_3 = \{2,4,6\} \subset \{2,4,6,8\} = D_4.
\end{equation*} 	  
But, already by considering a $3$-regular graph this property is no longer satisfied. We consider one example in Figure \ref{fig:comparison_regular_3_graph}. Inclusion of the vertex sets would imply that by adding an additional vertex to the sub-graph we would not only increase the number of possible configuration but can also construct from more vertices always configurations in such a way that they have the same density as some configuration on less vertices. 
\begin{figure}[!htb]
	\centering
	\begin{subfigure}{0.32\textwidth}
		\centering
		\includegraphics[scale=0.1]{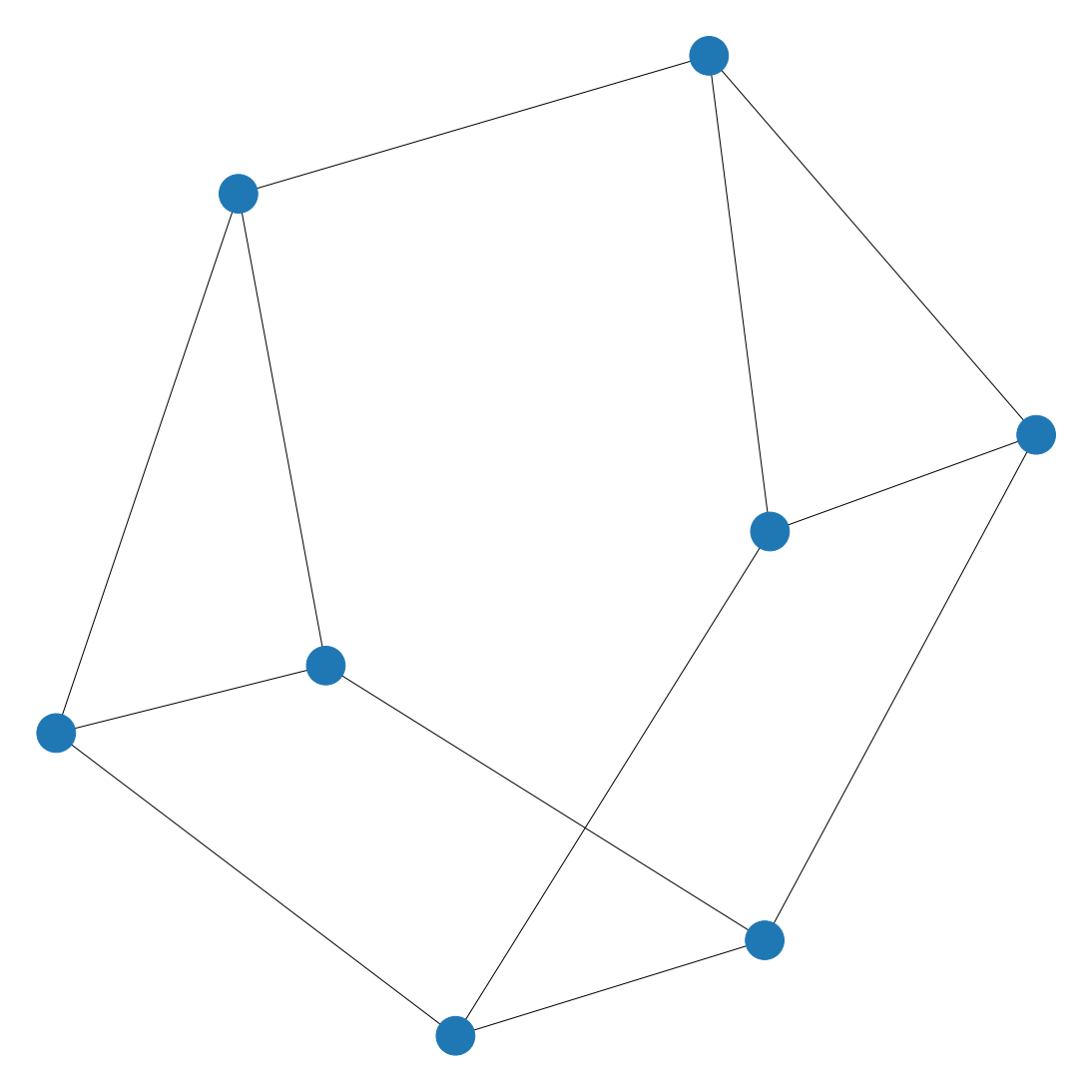}
	\end{subfigure}\hfill
	\begin{subfigure}{0.32\textwidth}
		\centering
		\includegraphics[scale=0.1]{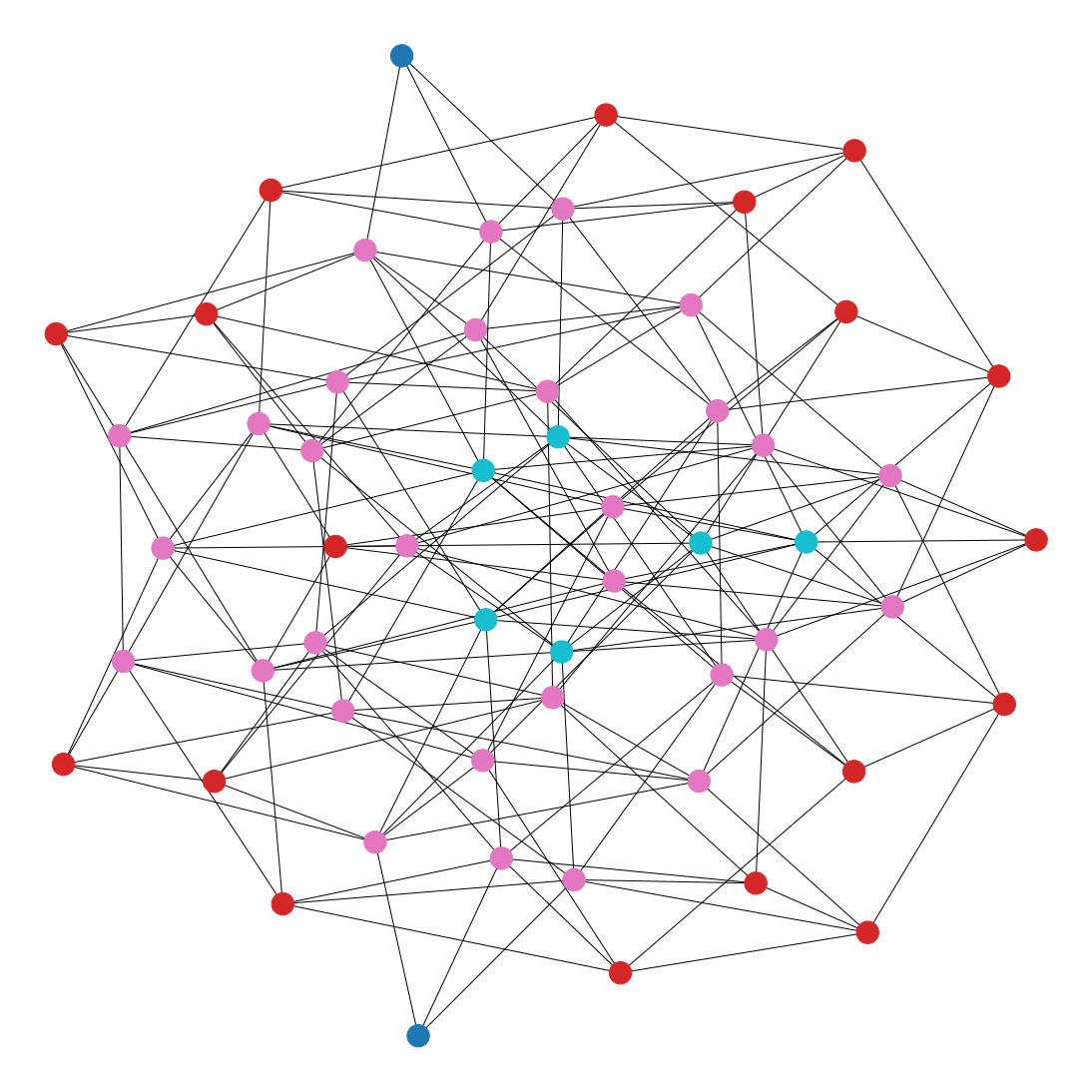}
	\end{subfigure}\hfill
	\begin{subfigure}{0.32\textwidth}
	 	\centering
		\includegraphics[scale=0.1]{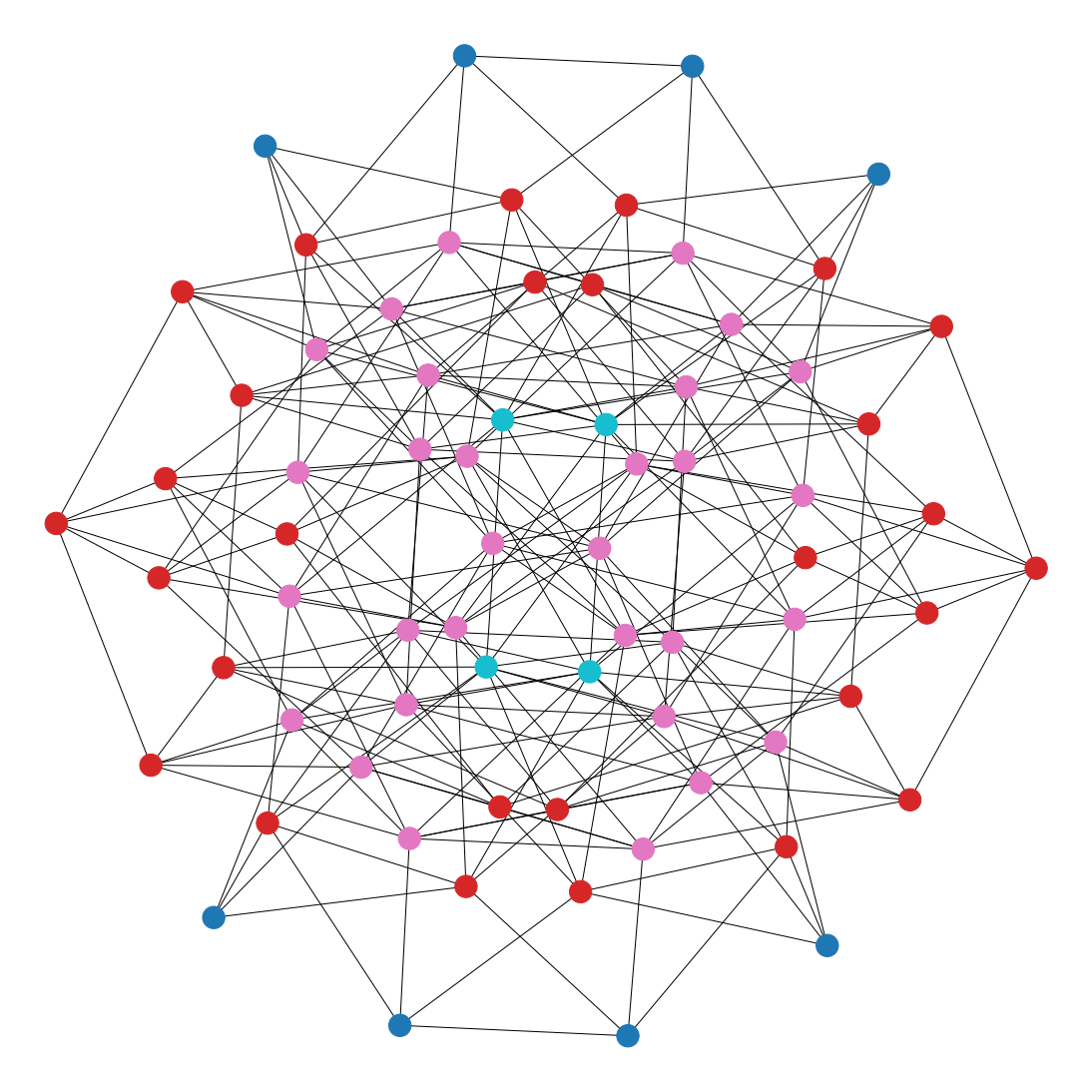}
	\end{subfigure}
	\caption{From left to right the underlying $3$-regular graph on $8$ vertices, the associated graphs $\mathfrak{L}_3$ and $\mathfrak{L}_4$. In this case we encounter $D_3\cap D_4 = \emptyset$.\label{fig:comparison_regular_3_graph}}
\end{figure}
It is intuitively understandable that this cannot be possible if the degree of the underlying graph $L$ is greater than $2$ and sufficiently many vertices are used to span the sub-graph. Indeed, in the case of Figure \ref{fig:comparison_regular_3_graph} we obtain for the degree sets 
\begin{equation}
	D_3 = \{3,5,7,9\} \not\subset \{4,6,8,10\} = D_4
\end{equation}
and even $D_3\cap D_4 = \emptyset$. Nonetheless, we can find implications of degrees in $\mathfrak{L}_k$ for any $k$ exploiting the symmetry of the binomial coefficient and the construction via the symmetric difference which yields, in particular, a symmetry for $k$ and $\bar{n}-k$ densest sub-graphs.
\begin{proposition} \label{prop:dense_subgraphs_k_and_nk}
	Let $k\in\left\{1,\hdots,\bar{n}-1\right\}$ and let $\mathfrak{v} \in \mathfrak{V}_k$ such that $L_{\mathfrak{v}}$ defines a densest vertex induced sub-graph on $k$ vertices in $L$. Then $\mathfrak{v}^c:=V\backslash \mathfrak{v}$ defines a densest sub-graph in $L$ on $\bar{n}-k$ vertices.\\
	Moreover, let $\mathfrak{v} \in \mathfrak{V}_k$ such that $L_{\mathfrak{v}}$ defines a densest vertex induced sub-graph on $k$ vertices in $L$. Then, the bipartite sub-graph $L_{\mathfrak{v},\mathfrak{v}^c}=((\mathfrak{v},\mathfrak{v}^c), E_{\mathfrak{v},\mathfrak{v}^c}) $ satisfies 
	\begin{equation*}
		|E_{\mathfrak{v},\mathfrak{v}^c}| = \min_{\mathfrak{w}\in\mathfrak{V}_k} |E_{\mathfrak{w},\mathfrak{w}^c}|
	\end{equation*}
\end{proposition}
\begin{proof}
	The proof of the first claim follows by the previously observed structure of the degrees of $\mathfrak{v}$ in $\mathfrak{L}_k$ and the property $D_k = D_{\bar{n}-k}$.
	\begin{align*}
		(\bar{n}-k)\bar{d}-2|E_{\mathfrak{v}^c}| &= \mathrm{deg}_{\bar{n}-k}(\mathfrak{v}^c) = \mathrm{deg}_k(\mathfrak{v}) = \min_{\mathfrak{w}\in\mathfrak{V}_k} \bar{d}k-2|E_{\mathfrak{w}}|\\
		&=\min_{\bar{\mathfrak{w}}\in\mathfrak{V}_{\bar{n}-k}} (\bar{n}-k)\bar{d}-2|E_{\bar{\mathfrak{w}}}| = (\bar{n}-k)\bar{d}-2\max_{\bar{\mathfrak{w}}\in\mathfrak{V}_{\bar{n}-k}}|E_{\bar{\mathfrak{w}}}|.
	\end{align*}
	Turning to the second claim, we observe by symmetry of the densest sub-graphs shown in the first claim, for $\mathfrak{v}$ a densest sub-graph, that 
	$|E_{\mathfrak{v}}|+|E_{\mathfrak{v}^c}|=\max_{\mathfrak{w}\in\mathfrak{V}_k}(|E_{\mathfrak{w}}|+|E_{\mathfrak{w}^c}|)$ and, consequently,
	\begin{equation*}
		|E_{\mathfrak{v},\mathfrak{v}^c}| = |E|-(|E_{\mathfrak{v}}|+|E_{\mathfrak{v}^c}|) = \min_{\mathfrak{w}\in\mathfrak{V}_k}(|E|-(|E_{\mathfrak{w}}|+|E_{\mathfrak{w}^c}|)) = \min_{\mathfrak{w}\in\mathfrak{V}_k}|E_{\mathfrak{w},\mathfrak{w}^c}|.
	\end{equation*}		 
\end{proof}
More symmetries of $\mathfrak{L}_k$ and links to $k$ induced sub-graphs of some underlying can be deduced from the previous properties. In particular, a link between the number of sub-graphs on $k$ and $\bar{n}-k$ vertices containing a fixed number of edges can be established.
\begin{proposition}\label{prop:formula_degree_size_tilde_H_k_r}
	Let $l\in D_k$, $\mathfrak{V}_{D_k;l}:=\{\mathfrak{v}\in \mathfrak{V}_k|\mathrm{deg}_k(\mathfrak{v})=l\}$ and $\mathfrak{L}_{k;l'}$ for $l'\in\mathbb{N}$ the set of vertex induced sub-graphs of $L$ on $k$ vertices containing exactly $l'$ edges. Then 
	\begin{equation}
		\left|\mathfrak{V}_{D_k;l}\right|=\left|\mathfrak{L}_{k;\frac{k\bar{d}-l}{2}}\right|.
	\end{equation}
	Furthermore,
	\begin{equation}
		\left|\mathfrak{L}_{k;\frac{k\bar{d}-l}{2}}\right| = \left|\mathfrak{L}_{\bar{n}-k;\frac{(\bar{n}-k)\bar{d}-l}{2}}\right|.
	\end{equation}
\end{proposition}
\begin{proof}
	Consider the identity given in Proposition \ref{prop:degree_formula_in_tilde_H_k_r} for $\mathfrak{v}\in\mathfrak{V}_k$ and let $l\in D_k$. Then 
	\begin{equation*}
		|\mathfrak{V}_{D_k;l}| = |\{\mathfrak{v}\in\mathfrak{V}_k|\mathrm{deg}_k(\mathfrak{v})= l\}| = \left|\left\{\mathfrak{v}\in\mathfrak{V}_k\bigg|\;|E_{\mathfrak{v}}| = \dfrac{k\bar{d}-l}{2} \right\}\right|.
	\end{equation*}
	Moreover, the set $\mathfrak{V}_k$ contains all subsets $\mathfrak{v}\subset V$ of size $k$. Hence, 
	\begin{equation*}
		\left|\left\{\mathfrak{v}\in\mathfrak{V}_k\bigg|\;|E_{\mathfrak{v}}| = \dfrac{k\bar{d}-l}{2} \right\}\right|= |\mathfrak{L}_{k;\frac{k\bar{d}-l}{2}}|
	\end{equation*} 
	and the first claim follows. The second claim follows by $\mathfrak{L}_k\stackrel{\sim}{=}\mathfrak{L}_{\bar{n}-k}$ and the identity given in Proposition \ref{prop:formula_degree_size_tilde_H_k_r}.
\end{proof}
While the preceding results give insights in the links between $\mathfrak{L}_k$ and the extremal properties of $k$-sub-graphs of some underlying $\bar{d}$-regular graph, we are also interested in the local combinatorial properties of vertices in a $k$-sub-graph. This will help us quantify more properties of $\mathfrak{L}_k$ as well as establish the bases for results on Markov chains on $\mathfrak{L}_k$ induced by certain types of exclusion processes. The following result gives a bound on the average degree and the density of vertices in some sub-graph $L_{\mathfrak{v}}$ induced by a vertex $\mathfrak{v}\in\mathfrak{V}_k$.
one can also wonder about $\big||E_{\mathfrak{v}}|-|E_{\mathfrak{v}}^c|\big|$ as a function of $\bar{n},\bar{d}$ and $k$. It turns out that this is in fact a constant for $\bar{d}$-regular graphs.
\begin{lemma}\label{lem:avg_deg_difference_edge_number_v_v_c}
	Let $L$ be a $\bar{d}$-regular graph on $\bar{n}$ vertices and $k\in\{1,\hdots,\bar{n}-1\}$. Then, for any $\mathfrak{v}\in\mathfrak{V}_k$, we have
	\begin{equation}
		\mathrm{avg\;deg}_{\bar{n}-k} (L_{\mathfrak{v}^c}) = \dfrac{\bar{d}(\bar{n}-2k)+k\,\mathrm{avg\;deg}_{k} (L_{\mathfrak{v}})}{\bar{n}-k}
	\end{equation}
	and
	\begin{equation}
		\big||E_{\mathfrak{v}}|-|E_{\mathfrak{v}^c}|\big| = \left|\dfrac{\bar{d}(\bar{n}-2k)}{2}\right|.
	\end{equation}		
\end{lemma}
\begin{proof}
	Using the isomorphism between $\mathfrak{L}_k$ and $\mathfrak{L}_{\bar{n}-k}$ we obtain that for $\mathfrak{v}\in\mathfrak{V}_k$ due to the equality of the degrees of $\mathfrak{v}$ and $\mathfrak{v}^c$
	\begin{equation*}
	k(\bar{d}-\mathrm{avg\;deg}_{k} (L_{\mathfrak{v}}))=\mathrm{deg}_k(\mathfrak{v})=\mathrm{deg}_{\bar{n}-k}(\mathfrak{v^c})=(\bar{n}-k)(\bar{d}-\mathrm{avg\;deg}_{\bar{n}-k} (L_{\mathfrak{v}^c}))
	\end{equation*}
	and, equivalently, 
	\begin{equation*}
		\mathrm{avg\;deg}_{\bar{n}-k} (L_{\mathfrak{v}^c}) = \dfrac{\bar{d}(\bar{n}-2k)+k\,\mathrm{avg\;deg}_{k} (L_{\mathfrak{v}})}{\bar{n}-k}.
	\end{equation*}
	For the second claim, we use that $k\,\mathrm{avg\;deg}_{k} (L_{\mathfrak{v}})=2|E_{\mathfrak{v}}|$ such that by the first claim
	\begin{equation*}
		2(|E_{\mathfrak{v}^c}|-|E_{\mathfrak{v}}|) = (\bar{n}-k)\,\mathrm{avg\;deg}_{\bar{n}-k} (L_{\mathfrak{v}^c})-k\,\mathrm{avg\;deg}_{k} (L_{\mathfrak{v}}) = \bar{d}(\bar{n}-2k)
	\end{equation*}	  
	which yields the second claim.
\end{proof}
Indeed, some well-known results may be derived directly from the degree formula obtained in Proposition \ref{prop:degree_formula_in_tilde_H_k_r}. In particular, relationships between sub-graphs of $L$ and their complements with respect to $L$ become easily accessible.
\begin{lemma}\label{lem:least_dense_is_densest_in_complement}
	Let $L=(V,E)$ be a simple connected graph on $n$ vertices. Assume that $\mathfrak{v}\subset V$ is a least dense $k$-sub-graph of $L$. Then $\mathfrak{v}$ is a densest $k$-sub-graph of $L^c$. 
\end{lemma}
\begin{proof}
	Let $\mathfrak{v}\subset V$, $|\mathfrak{v}|=k$ and assume that $(\mathfrak{v},E_{\mathfrak{v}})$ is a least dense sub-graph in $L$. Then, 
	\begin{equation}
		|E_{\mathfrak{v}}| = \min_{\mathfrak{w}\subset V,|\mathfrak{w}|=k}|E_{\mathfrak{w}}| = \dfrac{k(k-1)}{2}-\max_{\mathfrak{w}\subset V,|\mathfrak{w}|=k}|E_{\mathfrak{w}}^c|.
	\end{equation}
	The property $|E_{\mathfrak{v}}|+|E_{\mathfrak{v}}^c|=\frac{k(k-1)}{2}$ yields the claim.
\end{proof}
While Lemma \ref{lem:least_dense_is_densest_in_complement} gives a result on densest sub-graphs and their interpretation in the complement graph one has to wonder about the structure of said complement. For example, assuming that $L$ and $L^c$ are connected brings necessary additional properties of $L$ and $L^c$ with it which, in turn, imply more structure also on $\mathfrak{L}_k$. We turn now to geometric properties of $\mathfrak{L}_k$.

%% file: chapters/kPG_Outlook.tex
Finding the vertex connectivity of a graph is a well known problem, for which fast algorithms have been developed over the years. Theorem \ref{thm:L_k_vertex_connectivity} links this topic to the topic of finding the density of the densest $k$ sub-graph of $L$. In \cite{Henzinger:99361} the authors present an algorithm for finding the connectivity of a graph which is asymptotical as $\mathcal{O}(nm)$ to the number of vertices $n$ and the number of edges $m$. This implies that there is an algorithm for finding the densest $k$ vertex induced subgraph which is asymptotical as $\mathcal{O}\left(\binom{\bar{n}}{k}^2\cdot\dfrac{\bar{d}k(\bar{n}-k)}{2(\bar{n}-1)}\right)$. Further research into this link might yield additional insights on upper bounds of the complexity of finding dense sub-graphs of regular graphs.\par
Evidently, the analysis we performed in this section is not complete in any regard. Many graph theoretical questions remain open for later research. This ranges from the diameter of $\mathfrak{L}_k$, which would, indeed, push the results in the later sections in terms of their quantitative dimension, to the characterization of the Eigenvalues of the graph Laplacian associated to $\mathfrak{L}_k$ as well as its link to $L$.
\begin{figure}[!htb]
	\centering
	\begin{tikzpicture}[scale = 0.5]
		\Vertex[x=-3,y=3,label={$L$},size = 1.5]{t1}
		\Vertex[x=3,y=3,label={$C$},size = 1.5]{t2}
		\Vertex[x=-3,y=-3,label={$\mathfrak{L}_k$},size = 1.5]{t3}
		\Vertex[x=3,y=-3,label={$\mathcal{J}(\bar{n},k)$},size = 1.5]{t4}
		\Edge[Direct, label={$\hookrightarrow$}](t1)(t2)
		\Edge[Direct, label={$\Psi_k$}](t1)(t3)
		\Edge[Direct, label={$\Psi_k$}](t2)(t4)
		\Edge[Direct, label={$\hookrightarrow$}](t3)(t4);
	\end{tikzpicture}
	\caption{Link between types of graphs presented in this work. We denote by $\Psi_k$ the construction of the kPG associated to $L$. Research questions on one may be considered in the setting of another along the depicted arrows.\label{fig:diagram_L_Lk_Jnk}}
\end{figure}
Even conjectures remain at this point out of reach. Considering the diameter, from the isometry $\mathfrak{L}_k\stackrel{\sim}{=}\mathfrak{L}_{\bar{n}-k}$ we can, nonetheless, look for a formula which is symmetric in $k$ and $\bar{n}-k$. On the other hand, it will be constrained by the diameter of the underlying graph $L$. Employing the particle view, any particle in a configuration $\mathfrak{v}\in\mathfrak{V}_k$ has to move at most the diameter of $L$ to its new location in another configuration, if looking for the diameter of $\mathfrak{L}_k$. Therefore, a natural upper bound for $\mathrm{diam}(\mathfrak{L}_k)$ is $k\mathrm{diam}(L)$. It turns out that this bound is a lot bigger than the actual diameter of $\mathfrak{L}_k$ in cases testable via simulations. \\
Further topics include in particular the link between $L$ and $\mathfrak{L}_k$ as well as $\mathfrak{L}_k$ and $\mathcal{J}(\bar{n},k)$. Indeed, this gives the diagram represented in Figure \ref{fig:diagram_L_Lk_Jnk} which we already discussed briefly in the discussion. Sub-graph relationships are preserved via the horizontal arrows, while the spanning of a kPG based on some underlying graph $L$ is given in the vertical direction via the map $\Psi_k$. The potential of jumping between these images has been shown throughout this section and will, so does the author hope, nourish the research on graph theoretic problems related to sub-graphs. 